\documentclass[12pt,reqno]{amsart}
\usepackage{amssymb,amsmath,amsthm}
\usepackage{enumerate}

\DeclareMathOperator{\Ran}{Ran}

\DeclareMathOperator{\Ker}{Ker}

\DeclareMathOperator*{\slim}{s-lim}

\DeclareMathOperator{\diag}{diag}
\newcommand{\ess}{\text{\rm ess}}
\newcommand{\ac}{\text{\rm (ac)}}
\newcommand{\const}{\text{\rm const}}

\renewcommand\Im{\text{\rm Im}\,}
\renewcommand\Re{\text{\rm Re}\,}
\newcommand{\abs}[1]{\lvert#1\rvert}
\newcommand{\babs}[1]{\pmb\lvert#1\pmb\rvert}

\newcommand{\norm}[1]{\lVert#1\rVert}

\newcommand{\q}{\quad}
\newcommand{\wt}{\widetilde}
\newcommand{\wh}{\widehat}

\newcommand{\R}{{\mathbb R}}
\newcommand{\C}{{\mathbb C}}

\newcommand{\calH}{{\mathcal H}}
\newcommand{\calK}{{\mathcal K}}

\newcommand{\calN}{\mathcal{N}}

\newcommand{\eps}{\varepsilon}

\numberwithin{equation}{section}

\theoremstyle{plain}
\newtheorem{theorem}{\bf Theorem}[section]
\newtheorem{lemma}[theorem]{\bf Lemma}
\newtheorem{proposition}[theorem]{\bf Proposition}
\newtheorem{assumption}[theorem]{\bf Assumption}
\newtheorem{corollary}[theorem]{\bf Corollary}

\theoremstyle{definition}
\newtheorem{definition}[theorem]{\bf Definition}

\theoremstyle{remark}
\newtheorem*{remark*}{\bf Remark}

\let\goth\mathfrak
\let\cal\mathcal

\begin{document}

\title[A multichannel scheme]{A multichannel scheme in smooth scattering theory}
\author{Alexander Pushnitski}
\address{Department of Mathematics,
King's College London,
Strand, London, WC2R~2LS, U.K.}
\email{alexander.pushnitski@kcl.ac.uk}

\author{Dmitri Yafaev}
\address{Department of Mathematics, University of Rennes-1,
Campus Beaulieu, 35042, Rennes, France}
\email{yafaev@univ-rennes1.fr}


\begin{abstract}
We develop the   scattering theory for a pair  of self-adjoint operators $A_{0}=A_{1}\oplus\cdots \oplus A_{N}$ and $A=A_{1}+\cdots +A_{N}$ under the assumption that all pair products $A_{j}A_{k}$ with $j\neq k$ satisfy certain regularity conditions. Roughly speaking, these conditions mean that 
the products $A_{j}A_{k}$, $j\neq k$, can be represented as integral operators with smooth kernels in the spectral representation of the operator $A_{0}$.  We show that the absolutely continuous parts of the operators $A_{0}$ and $A$ are unitarily equivalent. This yields a smooth version of Ismagilov's theorem known earlier in the trace class framework. We also prove that the singular continuous
spectrum of the operator $A$ is empty and that its eigenvalues may accumulate only to  ``thresholds"  of the absolutely continuous
spectra of the operators $A_{j}$. Our approach relies on a system of resolvent equations which can be considered as a  generalization of Faddeev's equations for  three particle quantum systems.
\end{abstract}



\maketitle

\section{Introduction} \label{sec.a}

\subsection{Trace class and smooth methods}
  The methods used in scattering theory are naturally  divided into two groups:
 trace class and smooth. The fundamental result of the trace class method  is the Kato -- Rosenblum theorem     which establishes the existence   of wave operators for a pair of self-adjoint operators $A_{0}$ and $A$ under the assumption that the perturbation  $A-A_0$ belongs to the trace class ${\goth S}_1$. In particular, the absolutely continuous (a.c.) parts of $A_0$ and $A$ turn out to be unitarily equivalent
to each other. 

D.~Pearson extended the Kato -- Rosenblum theorem to operators $A_{0}$ and $A$ 
acting in different Hilbert spaces.
More precisely, assuming that $A_0$ acts in $\calH_0$, $A$ acts in $\calH$ and $J:\calH_0\to\calH$
is a bounded operator (``identification'') such that the effective perturbation $AJ-JA_0$ is trace class, 
Pearson proved that the wave operators (see definition \eqref{eq:WO} below) 
for the triple $A_0,A,J$ exist.
Earlier, sufficient conditions which ensure 
both\footnote{Under Pearson's assumptions the wave operators are not necessarily isometric.}
the existence and isometricity of wave operators
in the two space setting were found by A.~L.~Belopol$'$ski\u{\i} and M.~Sh.~Birman.

The smooth method in scattering theory (see, e.g., the original papers  \cite{F2, Kuroda} or the book  \cite{Yafaev}) 
is heavily based on the explicit spectral analysis of the unperturbed operator $A_{0}$. This approach requires that the perturbation, 
which we will denote by $A_\infty$, be sufficiently ``regular" (smooth) 
in the spectral representation of the operator $A_{0}$.  
This allows one to deduce information about the operator $A=A_{0} +A_\infty$ from an equation relating the resolvents of the operators $A_{0}$ and $A$.  
Besides the existence and completeness of the wave operators for the pair $A_0$, $A$, 
the smooth method also yields nontrivial information about the singular component of the operator $A$. Normally, one proves that the singular continuous spectrum of $A$ is empty and that the eigenvalues of $A$ have finite multiplicities and can accumulate only to some exceptional spectral points (thresholds) of the operator $A_{0}$.

The standard scheme of smooth   theory works for $\calH_0=\calH$ and  $J=I$. In the present paper we extend it (under rather special assumptions) to the two spaces framework.

\subsection{Ismagilov's theorem and its smooth version}

In \cite{Isma}, R.~S.~Ismagilov found an important generalization of the Kato -- Rosenblum theorem. 
He considered the operator
\begin{equation}
A=A_\infty+A_1+\cdots+A_N 
\label{1.2a}
\end{equation}
where $A_\infty$ and $A_{j}$ are bounded self-adjoint operators in a Hilbert space $\cal H$  
and $A_\infty \in {\goth S}_1$, $A_{j}A_{k} \in {\goth S}_1$ for all $j\neq k$. 
The result of paper  \cite{Isma} can be formulated as follows. 
Consider the operator
\begin{equation}
A_0 =0\oplus A_1\oplus\cdots\oplus A_N 
\label{1.2b}
\end{equation}
in the Hilbert space
$$
\calH_0=\calH^{N+1}
\overset{\text{def}}{=}
\underbrace{\calH\oplus\dots\oplus\calH}_{N+1}
$$
i.e. $\calH_0$ is the direct sum of $N+1$ copies of $\calH$.  
Then  the a.c.\  parts  of the  
operators $A_{0}$ and $A$ are unitarily equivalent to each other. 
The scattering theory for this pair $A_{0}$,  $A$ was constructed later in the papers \cite{Su} by A.~V.~Suslov  and \cite{HoKa} by J.~S.~Howland and T.~Kato where it was shown that the    wave operators for $A_{0}$, $A$ exist and are complete. Strictly speaking, in \cite{Isma, Su,HoKa} the case   $A_\infty=0$, $N=2$ was considered, but the general case reduces easily to this one.

Our goal is to develop the   scattering theory for    the operators $A_{0} $ and $A $ 
given by \eqref{1.2a} and \eqref{1.2b}
under the assumption that $A_\infty$ and all pair products $A_{j}A_{k}$ with $j\neq k$ satisfy conditions of ``smooth" type.  In particular, we show that the a.c. parts of the operators $A_{0}$ and $A$ are unitarily equivalent. This yields a smooth version of Ismagilov's theorem. We also prove that the singular continuous
spectrum of the operator $A$ is empty and that its eigenvalues may accumulate only to 
the thresholds   of the operators $A_{j}$.

The unitary equivalence of the a.c.\  parts of $A_0$, $A$, of course, means that the a.c.\  spectrum of $A$ 
consists of the union of the a.c.\  spectra of $A_1$,\dots, $A_N$, and the a.c.\  subspace of $A$ has the structure of the orthogonal sum of the subspaces corresponding to the operators $A_j$. In the terminology of the multiparticle scattering, each operator $A_j$ contributes its
own channel to the a.c.\  subspace of $A$. Because of this, we refer to the scattering problem 
\eqref{1.2a}, \eqref{1.2b} for $N\geq2$ as a \emph{multichannel problem}, 
and call the case $N=1$ the single channel case.

In the forthcoming paper \cite{PY2} we apply our results to the spectral analysis 
of the operators of 
the type $\theta(H)-\theta(H_{0})$, where $H_{0}$, $H$ is a pair of operators 
for which the assumptions of the smooth scattering 
theory are satisfied, and $\theta$ is a piecewise continuous function.
We are also planning to
apply our results to the spectral theory 
of Hankel operators with piecewise continuous symbols $\theta$. 
In the trace class framework,  the latter problem was considered by J.~S.~Howland in  \cite{Howl0}.
In both cases, we represent the piecewise continuous function $\theta$ as a sum 
\begin{equation}
\theta=\theta_\infty+\theta_1+\cdots+\theta_N,
\label{1.2c}
\end{equation}
where  
$\theta_j$ are certain ``standard'' functions, each of which has only one discontinuity, 
and $\theta_\infty$ is a smooth function. 
Each $\theta_j$ gives rise to a ``simple'' explicitly diagonalisable 
self-adjoint operator $A_j$  and thus \eqref{1.2c} produces the decomposition \eqref{1.2a}  for $A=\theta(H)-\theta(H_{0})$. 

\subsection{A two spaces setup}\label{sec.a2}

Actually, we consider the problem in a more general framework where self-adjoint operators $A_{0}$ and $A$ act in different Hilbert spaces $\calH_{0}$ and $\calH$, respectively,  and a bounded operator (``identification") $J :\calH_{0}\to \calH $ is given. 
 We do not assume $J$ to be either isometric or invertible. 
For simplicity, we suppose that the operators $A_0$ and $A$  are  bounded.
Under some  assumptions, we construct a variant of smooth scattering 
theory for the triple $A_0$, $A$, $J$. 
This generalises the scheme well known for the case
$\calH=\calH_0$, $J=I$. Our assumptions on the ``abstract'' triple $A_0$, $A$, $J$
are rather special and are tailored in such a way that they are satisfied
by the ``concrete'' operators \eqref{1.2a}, \eqref{1.2b} 
and the identification 
\begin{equation}\begin{split}
J:\calH_0 &\to\calH,
\\
\mathbf{f}=(f_\infty,f_1,f_2,\dots,f_N)^\top
&\mapsto 
J\mathbf{f}=f_\infty +f_1+f_2+\dots+f_N.
\end{split}
\label{1.2e}\end{equation}
Our main assumption on the ``abstract'' triple $A_0$, $A$, $J$ is that the effective
perturbation $AJ-JA_0$ can be factorised as 
\begin{equation}
AJ-JA_0=JT 
\label{bb2}
\end{equation}
where $T$ is a bounded operator in $\calH_0$ and the operator $T^2$ is compact. Moreover, the operator  $TA_{0}$ (but not $T$ itself)   possesses some smoothness
properties with respect to $A_0$. 

  For the triple \eqref{1.2a}, \eqref{1.2b}, \eqref{1.2e}, 
it can be easily checked by a direct inspection that
factorisation \eqref{bb2} holds with  the operator $T$    given in the space 
$\calH_0 =\cal H^{N+1}$ by the matrix
\begin{equation}
T=  
\begin{pmatrix}
{A_\infty}& {A_\infty}& {A_\infty}&\ldots& {A_\infty}
\\
A_1& 0&A_1&\ldots&  A_{1}
 \\
 A_{2}& A_2 &0&\ldots&  A_{2}
 \\
 \vdots&  \vdots&\vdots& \ddots&\vdots &
  \\
 A_{N}& A_{N} & A_N&\ldots  & 0
 \end{pmatrix} .
\label{eq:AA}
\end{equation}

The importance of the factorisation \eqref{bb2} is as follows.
Subtracting $zJ$ from both sides, we get
$
(A-z)J=J(A_0-z)+JT
$
for any $z\in\C$. 
From here, using the notation
\begin{equation}
R(z)=(A-z)^{-1}, \quad R_0(z)=(A_0-z)^{-1},
\label{1.2h}
\end{equation}
we immediately obtain the resolvent identity
\begin{equation}
R(z)J(I+TR_0(z))=JR_0(z),
\quad 
\Im z\not=0.
\label{bb3}
\end{equation}
The key point of our construction is that we consider
\eqref{bb3} as an equation for $R(z)J$ (rather than for $R(z)$). 
Under our assumptions,   \eqref{bb3} turns out to be a Fredholm equation. 
Using this fact, we establish   the limiting absorption principle for $A$. 
That is, we prove that $R(z)J$ has boundary values in an appropriate sense when 
$z$ approaches the real axis staying away from the eigenvalues of $A$. Roughly speaking, our analysis   hinges on 
inverting the operator $I+TR_0(z)$ in \eqref{bb3}. 
Note that this operator acts in  $\calH_0$, and so 
we are able to carry out much of the analysis in the framework of a single
Hilbert space $\calH_0$ rather than in a pair of spaces $\calH_0$, $\calH$.

Analytic results obtained in this way allow us to verify the assumptions  of smooth scattering theory for the operators $A_{0}$,  $A$ and the auxiliary identification  $\wt{J}= J A_{0} $. Note that the identification $\wt J$ was used previously in \cite{HoKa} for the construction of the scattering  theory for the pair \eqref{1.2a}, \eqref{1.2b} under Ismagilov's assumptions. We work in the framework of a local version of scattering
theory, that is, all our smoothness assumptions are made for the 
spectral parameter $\lambda$ confined to an open bounded interval $\Delta\subset\R$
with $0\notin\Delta$. We first construct local wave operators for the triple $A_0$, $A$, $\wt{J}$ and the  interval $\Delta$.
Then using that $0\not\in\Delta$, we replace $\wt{J}$ by the original identification $J$.
We show that the  local wave operators for the triple $A_0$, $A$, $J$ are isometric and complete. By the standard density arguments, global spectral results 
can be easily derived from the local ones, if necessary. We also show that the singular continuous
spectrum of the operator $A$ is empty and that its eigenvalues in $\Delta$ do not have interior points of accumulation.

It is important that all our results for the ``abstract'' triple $A_0$, $A$, $J$ 
apply to the ``concrete'' triple 
\eqref{1.2a}, \eqref{1.2b}, \eqref{1.2e}.

\subsection{Comparison with a three particle scattering problem}

Let us discuss the analogy between \eqref{1.2a} and   the three
particle Hamiltonian
\begin{equation}
H=H_{0}+ V_{1}+V_{2}+V_{3}
\quad
\text{ in $L^2(\R^{2d })$, $d\geq1$.}
\label{1.2i}
\end{equation}
Here $H_{0}=-\Delta$ is the operator of the total kinetic energy 
and  $V_{j}$ are potentials of interaction of pairs of particles 
(for example, $V_{1}$ corresponds to the interaction of the second and the third particles);
the motion of the center of mass is removed.
The  potentials $V_j$ do  not decay at infinity, but the products $V_{j}V_{k}$ with $j\neq k$  possess this property.  Thus, we have a formal analogy between \eqref{1.2a} and \eqref{1.2i} if $H_0=0$. 
In fact, this analogy goes further. 
Recall that every two particle Hamiltonian $H_{j}=H_{0}+ V_j$ yields its own  
channels of scattering to the  three particle system 
(provided the corresponding two particle subsystem has a point spectrum).
Similarly, in the problem \eqref{1.2a},
every operator $A_{j}$ contributes its own band of the a.c.\  spectrum 
to the spectrum of the operator $A$.
Furthermore,   the resolvent equation \eqref{bb3} is  algebraically similar to 
the famous Faddeev's equations \cite{Fadd} for the three particle quantum system. This is discussed in the Appendix.

Nevertheless, our problem preserves many   features of the two particle scattering. 
Indeed, the two particle scattering matrix  differs from the identity operator by a compact operator, 
and a result of this type remains true for the pair \eqref{1.2a}, \eqref{1.2b}.
This should be compared with the fact that, as observed by R.~Newton,  the singularities of the three particle scattering matrix are determined by the scattering matrices for all  two particle Hamiltonians $H_{j}$; see \cite[Section~14.2]{LNM}  for a discussion of this phenomenon.
In particular, the  scattering matrix  minus the identity operator is not compact in the three particle case.

\subsection{The structure of the paper} 
The basic definitions of scattering theory are formulated in the   abstract framework in Section~\ref{sec.b}.
Our  main results concerning the pair \eqref{1.2a}, \eqref{1.2b} are stated in Section~\ref{sec.bb} under Assumption~\ref{assu}. Section~\ref{sec.ccs} plays the central role. Here we ``forget'' about the nature of the operators $A_0$, $A$, $J$,  formulate a list of hypotheses (Assumption~\ref{ass2})
on an ``abstract'' triple $A_0,A,J$ and then   develop a version of smooth scattering theory.
  In short Section~\ref{sec.f}, we show that under Assumption~\ref{assu} these hypotheses   are satisfied for the triple $A_{0}$, $A$ and $J$ 
defined by \eqref{1.2a}, \eqref{1.2b} and \eqref{1.2e}. Then we
translate the ``abstract'' results of Section~\ref{sec.ccs}  back into the setting of the
operators $A_j$, $A$.   
In Section~\ref{sec.e} we briefly discuss  
the stationary representations for the scattering matrix and the wave operators.

 \section{Scattering theory in a two spaces setting}\label{sec.b}    
         
In this preliminary section, we collect the required general definitions and results 
from scattering theory. For the details, see, e.g.,   the book \cite{Yafaev}.

\subsection{Wave operators}

Let $A_{0}$ (resp.\ $A$) be a bounded self-adjoint operator 
in a Hilbert space ${\cal H}_{0}$ (resp.\ ${\cal H}$). 
Throughout the paper, 
we denote by $E_{0}(\cdot)$ and $E(\cdot)$ 
the projection valued spectral measures of $A_0$ and $A$
and by $P_{0}^{\ac}$ and $P^{\ac}$ 
the orthogonal projections onto the a.c.\  subspaces ${\cal H}_{0}^{\ac}$ and ${\cal H}^{\ac}$ of these operators, and set
$$
E_{0}^{\ac}(\Delta)=E_{0}(\Delta) P_{0}^{\ac}, 
\quad
E^{\ac} (\Delta  )= E (\Delta  )  P^{\ac}
$$
for $\Delta\subset\R$. 
We also use the notation $R_0(z)$, $R(z)$ for the resolvents of 
$A_0$, $A$, see \eqref{1.2h}.
We denote by $ \sigma_\ess (A)$ and
$ \sigma_p(A)$   the essential and point spectra of a self-adjoint operator $A$.
The class of   compact operators is denoted by ${\goth S}_{\infty}$.

The wave operators for the operators $A_{0}$,  $A$, a bounded operator $J:  {\cal H}_{0}  \to  {\cal H}$ and a bounded open  interval $\Delta$ are defined   by the relation
 \begin{equation}
W_{\pm} (A,A_{0}; J , \Delta) = \slim_{t\to\pm\infty} e^{iAt}J e^{- iA_{0}t}E_{0}^{\ac}(\Delta),
\label{eq:WO}\end{equation}
 provided this strong limit exists. It is easy to see that under the assumption of its existence the wave operator possesses the intertwining property
  \begin{equation}
W_{\pm} (A,A_{0}; J  , \Delta ) A_{0}= A W_{\pm} (A,A_{0};
J ,  \Delta ). 
\label{eq:Int}\end{equation}

We need the following elementary assertion.

\begin{lemma}\label{iso}
Suppose that the wave operator \eqref{eq:WO} exists.
 If for  a real valued function $\varphi$
\begin{equation}
  J^* J -\varphi (A_{0})  \in{\goth S}_{\infty},
\label{eq:W8h}\end{equation}
then
\begin{equation}
W_{\pm}^* (A,A_{0};  J , \Delta ) W_{\pm} (A,A_{0}; J, \Delta ) =\varphi (A_{0}) E_{0}^{\ac} (\Delta).
\label{eq:W8m}\end{equation} 
In particular, if 
$ J^* J- I \in{\goth S}_{\infty}$,
then the operator $W_{\pm} (A,A_{0}; J,\Delta)$ is isometric on the subspace 
$\Ran E_{0}^{\ac}(\Delta)$. 
\end{lemma}

\begin{proof}
Observe that
\begin{equation}
\| J e^{- iA_{0}t} f\|^2=  (( J^* J-\varphi (A_{0})) e^{- iA_{0}t} f,e^{- iA_{0}t} f) + (\varphi (A_{0})f,f).
\label{eq:jj}\end{equation} 
Let $f \in \Ran E_{0}^{\ac}(\Delta)$ and $t\to\pm\infty$. Then the l.h.s.\  of \eqref{eq:jj} tends to
$\|W_{\pm} (A,A_{0}; J,\Delta) f\|^2$. By assumption \eqref{eq:W8h}, $\|Ê( J^* J-\varphi (A_{0})) e^{- iA_{0}t} f\|\to 0$ as $| t| \to \infty$, and hence the first term in the r.h.s.\  of  \eqref{eq:jj} tends to zero as $| t| \to \infty$. Therefore passing in \eqref{eq:jj}  to the limit $t\to\pm\infty$, we get \eqref{eq:W8m}.
\end{proof}

The isometric wave  operator $W_{\pm} (A,A_{0}; J ,\Delta)$ is called  complete if
$$
\Ran  W_{\pm}  (A,A_{0}; J ,\Delta )
=
\Ran E^{\ac} (\Delta).
$$
If both wave operators $W_{\pm} (A,A_{0}; J, \Delta )$ and $W_{\pm} ( A_{0}, A; J^*  , \Delta )$ exist, then they are adjoint to each other.
 
By the usual density arguments, if the local wave operators $W_\pm(A,A_0;J,\Delta_n)$ 
exist,  are isometric and complete for a collection of intervals $\{\Delta_n\}$ such that $\R\setminus(\cup_n\Delta_n)$
has the Lebesgue measure zero, then the global wave operators ($\Delta=\R$) also exist,  are isometric and complete. 
We state and prove all our results in the local setting, bearing in mind that the corresponding global results
follow automatically.

\subsection{Relative smoothness}

Let us discuss sufficient conditions for the existence of the wave operators
in the framework of the smooth scattering theory.
Let $A_0$ be a self-adjoint operator in a Hilbert space $\calH_{0}$, and let
$Q_{0}$ be a bounded operator acting from $\calH_{0}$ to another Hilbert space $\calK$ 
(of course, the case $\calK=\calH_{0}$ is not excluded).    
One of the equivalent definitions of the $A_{0}$-smoothness 
of $Q_{0}$ on an interval  $\delta$ in the sense of Kato is given by the condition
$$
\sup_{\lambda\in \delta, \varepsilon\neq 0}
\| Q_{0}  ( R_{0} (\lambda+i\varepsilon) - R_{0} (\lambda-i\varepsilon))Q_{0} ^* \| <\infty, 
\quad R_{0}(z) =(A_{0}-z)^{-1}.
$$
Let us state the Kato-Lavine theorem. 
 
\begin{proposition}\label{KL}
Assume that the factorisation
$$
A J - J A_{0}=Q^*  Q_{0} 
$$
with bounded operators
$Q_0:\calH_0\to\calK$
and $Q:\calH\to\calK$ holds. 
Assume that  the  operators  $Q_{0}$ and $Q$ are smooth 
in the sense of  Kato  relative to $A_{0}$ and $A$, respectively, on every compact subinterval $\delta$ of an interval  $\Delta$. Then the wave operators $W_{\pm} (A,A_{0}; J  , \Delta )$ and $W_{\pm} ( A_{0}, A; J^*  , \Delta )$ exist and are adjoint to each other.
         \end{proposition}
         
Usually the smoothness of $Q_{0}$ with respect to the ``unperturbed" operator $A_{0}$ can be verified directly. 
A proof of $A$-smoothness of the operator $Q$ is more complicated. 
If $\calH_0=\calH$,  $J=I$ and $Q=BQ_{0}$ with a bounded operator $B$, 
then in applications the $A$-smoothness of $Q$  is 
often  deduced  from a sufficiently strong form  of $A_{0}$-smoothness of the operator $Q_{0}$ 
combined with some compactness arguments. This is  discussed in in Subsection~2.3.
           
Let us formulate a convenient strong form of relative smoothness. 
Suppose  that  the spectrum of the operator $A_{0}$ on $\Delta$ is a.c.\  and  has a constant (possibly infinite) multiplicity $k$. We consider   a unitary mapping  
\begin{equation}
 {  F }_{0}: \Ran E_{0} (\Delta) \rightarrow L^2(\Delta;{\goth h}_{0})
 =L^2(\Delta)\otimes{\goth h}_{0} , \q \dim {\goth h}_{0}=k,
\label{eq:DIntF}\end{equation}
of $ \Ran E_{0} (\Delta)$
onto the space of vector-valued functions of  $\lambda\in \Delta$  
with values in an auxiliary Hilbert space ${\goth h}_{0}$.  Assume that $F_0$ 
maps $A_0$ to the  operator   of multiplication  by $\lambda$, that is,
\begin{equation}
({  F }_{0} A_{0} f)(\lambda)=\lambda ({  F }_{0}   f)(\lambda), \q
f\in \Ran E_{0} (\Delta).
 \label{eq:FF1}\end{equation}
We denote by $\babs{\cdot}$ the norm in $\goth h_0$. Note that 
\begin{equation}
\frac{d(E_{0}(-\infty,\lambda)f,f)}{d\lambda}=\babs{F_{0} (\lambda) f}^2
\label{eq:FF}\end{equation}
for all $f\in \Ran E_{0} (\Delta)$ and almost all $\lambda\in\Delta$.
Along with
$L^2(\Delta;{\goth h}_{0})$ we consider the space 
$C^\gamma(\Delta;{\goth h}_{0})$, $\gamma\in (0,1]$, of H\"older
continuous vector-valued functions.
  
\begin{definition}\label{strsm}
 A bounded   operator $Q_{0} :     {\cal H}_{0}  \to \calK$ is called strongly $A_{0}$-smooth (with an exponent
$\gamma \in(0,1]$) on    $  \Delta$ 
  if there exists a unitary diagonalization $F_{0}$ of $A_{0}|_{\Ran E_{0} (\Delta)}$ 
  such that the  operator ${  F}_{0}   Q_{0}^*  $ maps  $\calK$   continuously into
$C^\gamma (\Delta; {\goth h}_{0})$, i.e.,  
$$
 {\pmb|} ( {  F}_{0}   Q_{0} ^*  g)(\lambda){\pmb|} \leq C\|g\| ,\quad
 {\pmb|} ({  F}_{0}    Q_{0} ^*  g )(\lambda)- ( {  F}_{0}   Q_{0}  ^*  g )(\mu){\pmb|} \leq C |\lambda-\mu|^\gamma \| g \|.
$$
Here the constant $C$ does not depend on $\lambda$ and $\mu$ in  compact subintervals  of  $\Delta$.
\end{definition}

 For a strongly $A_{0}$-smooth operator $Q_{0} $, the operator 
 $Z_{0} (\lambda;  Q_{0} ):\calK\rightarrow{\goth h}_{0}$,
defined by the relation
\begin{equation} 
Z_{0} (\lambda; Q_{0} ) g =({  F}_{0}  Q_{0} ^*  g )(\lambda),
\label{eq:ZG}\end{equation}
  is bounded and depends H\"older continuously on $\lambda\in \Delta$. 
  According to \eqref{eq:FF} and \eqref{eq:ZG} we have
 \begin{equation} 
  \frac{d(Q_{0} E_{0}(-\infty,\lambda)Q_{0} ^*  g , g)}{d\lambda}=\babs{ Z_{0} (\lambda; Q_{0} ) g }^2
 \label{eq:ZG1}\end{equation}
 so that, for a strongly $A_{0}$-smooth operator $Q_{0} $, this expression  depends H\"older continuously on $\lambda\in \Delta$. 
Therefore the following result is a direct consequence (see, e.g., \cite{Yafaev}) of the spectral theorem for $A_{0}$ and the Privalov lemma.
  
\begin{proposition}\label{Privalov}
If an   operator $Q_{0} $ is   strongly $A_{0}$-smooth   on $\Delta$ with  some exponent $\gamma\in (0,1)$, then the  operator-valued function $Q_{0}R_0(z) Q_{0} ^*  $ is
H\"older continuous in the operator norm 
$($with the same exponent $\gamma)$ for $\Re z\in \Delta$ and $\pm\Im z\geq 0$. 
\end{proposition}
  
According to this result the strong $A_{0}$-smoothness   of an  operator $Q_{0} $  implies its $A_{0}$-smoothness in the sense of Kato on each  compact subinterval $\delta$ of  the  interval $\Delta$.  
    
We will also need the following result (see, e.g., \cite[Theorem~1.8.3]{Yafaev}). 
It is known as  the analytic Fredholm alternative.  
  
\begin{proposition}\label{Fred}
Let ${\cal H}_{0}$ be a Hilbert space, and let $\Delta$ be an open interval.
Suppose that  the operator-valued function  $G_{0}(z):{\cal H}_{0} \to {\cal H}_{0}$ defined on the set $\Re z\in\Delta$, $\Im z\neq 0$,  is analytic,  the operators $G_{0}(z)^p$ are compact for 
some natural $p$ and the point $-1$ is not an eigenvalue of  the operators $G_{0}(z)$. Assume also  that  $G_{0}(z)$ is continuous up to the cut along the real axis. Then $(I+ G_{0}(z))^{-1}$ is a   continuous operator-valued function of $z$ for $\pm \Im z\geq 0$, $\Re z\in\Delta$, away from a closed set ${\cal N}_{\pm} \subset \Delta$ of measure zero. The set  ${\cal N}_{\pm}$ consists of the points $\lambda\in \Delta$ where  the   equation 
$$
 g+  G_{0}(\lambda\pm i0)  g=0
$$
has a nontrivial solution $ g\in  {\cal H}_{0}$.      Moreover,  $(I+ G_{0}(z))^{-1}$ is    H\"older continuous if $G_{0}(z) $ is    H\"older continuous.
 \end{proposition}
 
   \subsection{The ``single channel'' case}
   
Let us recall  the well known basic results 
(see, e.g., \cite{Kuroda} or  \cite[Sections~4.6, 4.7]{Yafaev}) in the ``single channel" setting
(i.e., for $N=1$ in \eqref{1.2a}) as it provides a simple model for the multichannel case considered in the next section. 
Let $A_0$ and ${A_\infty}$ be bounded self-adjoint operators in a Hilbert space 
$\cal H$ and let $A=A_0+{A_\infty}$ (this is consistent with notation \eqref{1.2a} if we set $N=1$ and $A_{1}=A_{0}$). Here  $A_0$ is regarded as the  ``free" operator,  
$A$ as the perturbed one, and $A_\infty$ is the perturbation.

\begin{proposition}\label{prp.b1} 
 Let $\Delta\subset\R$ be a bounded open interval; 
assume that the spectrum of $A_0$ on $\Delta$ is purely a.c.\ 
with a constant multiplicity.  Let $Q_{0}$ be a bounded
operator with $\Ker Q_{0}=\{0\}$; assume that $Q_{0}$ is strongly $A_0$-smooth on $\Delta$ with an 
exponent $\gamma>1/2$. 
Assume that the operator ${A_\infty}$ can be represented as ${A_\infty}=Q_{0}^*K  Q_{0}$ with a compact
operator $K $. Then the local wave operators 
$$
W_\pm(A,A_0;\Delta)
:=\slim_{t\to\pm\infty} e^{itA}e^{-itA_0} E_{0}(\Delta)
$$
exist and enjoy the intertwining property \eqref{eq:Int}. 
These operators are isometric on $\Ran E_{0}(\Delta)$ and are complete: 
$$
\Ran W_\pm(A,A_0;\Delta)=\Ran E^\ac(\Delta).
$$
The singular continuous spectrum of $A$ on $\Delta$ is absent.
All eigenvalues of $A$ in $\Delta$ have finite multiplicities and can 
accumulate only to the endpoints of $\Delta$. 
\end{proposition}
 
Proposition~\ref{prp.b1}, in particular, implies that the restriction of $A$ onto
$\Ran E^\ac(\Delta)$ is unitarily equivalent to the restriction of $A_0$ onto
$\Ran  E_{0}(\Delta)$. That is, the a.c.\  spectrum of $A$ on $\Delta$ 
has a constant multiplicity which coincides with the multiplicity of the a.c.\ 
spectrum of $A_0$ on $\Delta$. 

The key step in proof of Proposition~\ref{prp.b1} is the limiting absorption principle
in the following form. 

\begin{proposition}\label{prp.b1a}
Under the hypotheses of Proposition~$\ref{prp.b1}$
the operator valued function 
$Q_{0} R(z) Q_{0}^*$ is H\"older continuous in the operator norm
for $\Re z\in\Delta\setminus \sigma_p(A)$ and $\pm\Im z\geq 0$. 
\end{proposition}

The usual scheme of the proof of Proposition~\ref{prp.b1a} proceeds as follows. 
First, one uses  Proposition~\ref{Privalov} 
to establish the H\"older continuity of $Q_0R_0( z )Q_0^*$. 
Then, from the standard resolvent identity 
\begin{equation}
R(z)(I+A_\infty R_0(z))=R_0(z),
\label{2.0}
\end{equation}
using the factorisation $A_\infty=Q_0^*KQ_0$, one obtains
\begin{equation}
Q_0R(z)Q_0^*(I+KQ_0R_0(z)Q_0^*)=Q_0R_0(z)Q_0^*.
\label{2.1}
\end{equation}
This is a Fredholm equation for the operator $Q_0R(z)Q_0^*$. Applying now Proposition~\ref{Fred} to the operator valued function $G_{0}(z)= KQ_0R_0(z)Q_0^*$, we see that $Q_0R(z)Q_0^*$ is continuous for $\Re z\in\Delta$,  $\pm\Im z\geq 0$ away from the exceptional set  ${\cal N}_{\pm}$. It follows that the singular spectrum of the operator $A$ is contained in the set ${\cal N}_{+} \cap {\cal N}_-$. Under the assumption $\gamma>1/2$ it is possible to prove that ${\cal N}_{+} = {\cal N}_-$ and that this set consists of eigenvalues of $A$. Thus the singular continuous spectrum of the operator $A$ in $\Delta$ is empty. Basically the same arguments show that the  eigenvalues of $A$ do not have in $\Delta$ interior points of accumulation. 

It follows from Proposition~\ref{prp.b1a} that the operator $Q_{0}$ is $A$-smooth in the sense of Kato on every compact subinterval of $\Delta\setminus \sigma_p(A)$. Thus
Proposition~\ref{KL}  guarantees the existence of the wave operators $W_\pm(A,A_0;\Delta)$ and $W_\pm(A_0, A ;\Delta)$. This ensures that the wave operators $W_\pm(A,A_0;\Delta)$ are isometric and complete.

Assumption $\gamma>1/2$ in Proposition~\ref{prp.b1} is required only for the statements about the singular 
continuous and point spectra of $A$. Construction of the wave operators can be achieved
under the weaker assumption $\gamma>0$. 

Note that, in the case $N=1$, our resolvent equation \eqref{bb3} for the triple  \eqref{1.2a}, \eqref{1.2b}, \eqref{1.2e} reduces to \eqref{2.0}.

\section{Main results} \label{sec.bb}

\subsection{A gentle introduction: essential spectrum}

Weyl's theorem on the invariance of the essential spectrum of a self-adjoint operator
under compact perturbations can be regarded as    a precursor of scattering theory. Here we use this setting in order to illustrate the issues specific to our multichannel situation. The statement below is well known (see, e.g., the book \cite{Pe}). 
We give the proof   since it explains why the products
$A_j A_k$ appear in the analysis of $A$. 
Our proof   relies only on a direct construction of Weyl's sequences.

   \begin{proposition}\label{ess}
   Let  $A_{1}$, \ldots, $A_{N}$ be bounded self-adjoint operators
 such that $A_{j} A_{k}\in{\goth S}_{\infty}$ for $j\neq k$, and let the operator $A$ be defined by formula \eqref{1.2a}. Then
\begin{equation}
\sigma_{\ess}(A)=\bigcup_{j=1}^N \sigma_{\ess}(A_{j})  \q ({\rm mod}\{0\}),
\label{eq:aa}
\end{equation}
that is, the left and the right-hand sides coincide up to a possible zero eigenvalue of infinite multiplicity.
\end{proposition}
Of course, with notation \eqref{1.2b}, formula \eqref{eq:aa} can be equivalently rewritten as 
$$
\sigma_{\ess}(A)=\sigma_{\ess}(A_0) \q ({\rm mod}\{0\}).
$$
\begin{proof}
  If $\lambda\in \sigma_{\ess}(A_{j}) $ for some $j=1,\ldots, N$
  and $\lambda\not=0$, then there exists a (Weyl) sequence $f_{n}$ such that $\|f_{n}\|=1$, $f_{n}\to 0$ weakly and $g_{n}:=A_{j}f_{n}-\lambda f_{n}  \to 0$ strongly as $n\to \infty$. Since the operators $A_{k} A_{j}$, $k\neq j$, are compact, it follows that for all $k\neq j$
$$
A_{k}f_{n}=\lambda^{-1} (A_{k} A_{j} f_{n}-A_{k}g_{n})
$$   
converge  strongly to zero. 
Thus $f_{n}$ is also the Weyl sequence for the operator $A$ and the same $\lambda$.

   Conversely,
  if $\lambda\in \sigma_{\ess}(A ) $, then there exists a Weyl sequence $f_{n}$ such that $\|f_{n}\|=1$, $f_{n}\to 0$ weakly and $h_{n}:=A f_{n}-\lambda f_{n}  \to 0$ strongly as $n\to \infty$. Since the operators $A_{k} A_{j}$, $k\neq j$, are compact, this implies that,  for all indices $j=1,\ldots, N$, the sequences
   \begin{equation}
  ( A_{j} -\lambda )A_{j}f_{n} = A_{j} ( A- \sum_{k\neq j} A_{k} -\lambda)  f_{n}=
 -  \sum_{k\neq j}A_{j} A_{k} f_{n}+A_{j} h_{n}\to 0
 \label{eq:weyl}\end{equation}
       strongly  as $n\to \infty$. 
Observe that if     $\lambda\neq 0$, then at least for one of the indices $j=1,\ldots, N$, the norm
$\| A_{j}f_{n}\| $ does not tend to zero as $n\to\infty$.  
Indeed, supposing the contrary,  we find that
$$
\sum_{j=1}^N A_{j}f_{n}= Af_{n}=\lambda f_{n}+  h_{n}
$$
      converges  strongly to zero while the norm of the right-hand side tends to $|\lambda |\neq 0$.
  If $  A_{j}f_{n}  $ does not converge to zero, then we can assume that $ \| A_{j}f_{n} \|    \geq c> 0$ and set $\varphi_{n}=A_{j} f_{n} \| A_{j}f_{n} \|^{-1}$. It follows from \eqref{eq:weyl} that $ (A_{j} -\lambda) \varphi_{n}\to 0$ as $n\to \infty$. Moreover, $\varphi_{n}\to 0$ weakly  because $f_{n}\to 0$ weakly as $n\to \infty$. 
    \end{proof}

  It is easy to see that the condition  ``mod $\{0\}$" cannot be dropped in \eqref{eq:aa}.
    Indeed,  let $\calH$ be a Hilbert space of infinite dimension,  and let
    $A_{1}=\diag\{ I, 0\}$, $A_{2}=\diag\{ 0,I\}$
     be diagonal operators in the space $\calH\oplus\calH$. Then
   $\sigma_{\ess}(A_1)= \sigma_{\ess}(A_2)=\{0,1\}$ while $ \sigma_{\ess}(A_{1}+A_2)=\{1\}$.

\subsection{Multichannel scheme: main results}\label{sec.bb3}

Let $N\in\mathbb N$ and let ${A_\infty}$, $A_j$, $j=1,\dots,N$,  
be bounded self-adjoint operators in a Hilbert space $\cal H$. 
As in Section~\ref{sec.a}, we set
$$
A= A_\infty+A_1+A_2+\dots+A_N.
$$
Let $\Delta\subset \R$ be a bounded open interval with 
$0\not\in \Delta$, and let $X $ be a bounded operator in $\cal H$. 
We need

\begin{assumption}\label{assu}
\begin{enumerate}[\rm (i)]
\item
One has
$$
\Ker X =\Ker X ^*=\{0\}.
$$

\item
The spectra of the operators  $A_{1}$, \ldots, $A_{N}$
in $\Delta$ are a.c.\  and have constant  multiplicities. 
For all $j=1,\ldots, N$,  the operator $X $ is  strongly $A_{j}$-smooth 
$($see Definition~$\ref{strsm})$ 
on $\Delta$  with an exponent $\gamma>1/2$.

\item
The operator ${A_\infty}$ can be represented as
${A_\infty} =X ^* K_\infty X  $ with a compact operator $K_\infty$.

\item 
For all $j\neq \ell$,
the operators $A_{j} A_{\ell}$ can be represented as
\begin{equation}
A_{j} A_{\ell} =X ^*  K_{j,\ell} X 
\label{eq:m}\end{equation}
where the operators $K_{j,\ell}$ are compact.
\item
The operators $X   A_{j} X ^{-1}$ are bounded\footnote{Strictly speaking, we assume  that the operators $X   A_{j} X ^{-1}$ defined on the dense 
  set $\Ran X$ extend to bounded operators. The same convention applies to all other operators of this type.} for all $j=1,\ldots,N$.
\end{enumerate}
\end{assumption}
 
Our spectral results are collected in the following assertion.

\begin{theorem}\label{thm.b2}
Under Assumption~{\rm\ref{assu}} one has:
\begin{enumerate}[\rm (i)]
\item
The a.c.\  spectrum of $A$ 
on $\Delta$ has a constant multiplicity equal to the sum of the multiplicities
of the a.c.\  spectra of $A_j$, $j=1,\dots,N$, on $\Delta$. 
\item
The singular continuous spectrum of $A$ on $\Delta$ is empty. 
The eigenvalues of $A$ in $\Delta$ have finite multiplicities and can 
accumulate only to the endpoints of $\Delta$. 
\item
The operator valued function $X R(z) X^*$ 
is H\"older continuous $($in the operator norm$)$ in $z$ for $\pm\Im z\geq 0$
and $\Re z\in \Delta\setminus \sigma_p(A)$.
\end{enumerate}
\end{theorem}

The scattering theory for the set of operators $A_{1},\ldots, A_{N}$ and the operator $A$ is described in the following assertion. We denote by $E_j(\Delta)$ the spectral projection of $A_j$ corresponding to the interval $\Delta$.

\begin{theorem}\label{thm.b2b}
Under Assumption~$\ref{assu}$ one has:
\begin{enumerate}[\rm (i)]
\item
For all $j=1,\dots,N$, the local wave operators  
$$
W_\pm(A,A_j;\Delta) = \slim_{t\to\pm\infty} e^{iAt}   e^{- iA_{j}t}E_{j} (\Delta) 
$$
exist and enjoy the   intertwining property
$$
AW_\pm(A,A_j;\Delta)=W_\pm(A,A_j;\Delta)A_j.
$$
The operators $W_\pm(A,A_j;\Delta)$ are isometric on $\Ran E_{j}(\Delta)$.
Their ranges are orthogonal to each other:
\begin{equation}
\Ran W_\pm(A,A_j;\Delta)\perp \Ran W_\pm(A,A_\ell ;\Delta), 
\quad 
j\not= \ell.
\label{eq:Orth}\end{equation}

\item
The asymptotic completeness holds:
\begin{equation}
\bigoplus_{j=1}^N \Ran W_\pm(A,A_j;\Delta)
=
\Ran E^\ac(\Delta).
\label{eq:AsCo}\end{equation}
\end{enumerate}
\end{theorem}
We note that the first statement of Theorem~\ref{thm.b2} is a direct consequence of Theorem~\ref{thm.b2b}.

As already mentioned in the Introduction, in the trace class framework (Ismagilov's theorem) 
the statements of Theorem~\ref{thm.b2b} (but not parts (ii, iii) of Theorem~\ref{thm.b2}) 
are obtained in \cite{Howl0,Su} under the assumptions $A_jA_k\in{\goth S}_1$, $j\not=k$, and $A_\infty\in {\goth S}_1$. Part (i) of Theorem~\ref{thm.b2} goes back to \cite{Isma}.

\section{A two spaces setup}\label{sec.ccs}

\subsection{Assumptions and results}

Let $A_0$ (resp. $A$) be a bounded self-adjoint operator in a Hilbert space
$\calH_0$ (resp. $\calH$). 
Let $J:\calH_0\to\calH$ be a bounded operator (the ``identification''). 
Our key assumption is that the factorisation 
\eqref{bb2} holds true with a bounded operator $T$ in $\calH_0$. 
We fix an open bounded interval $\Delta\subset \R\setminus\{0\}$
and a bounded operator $Q_0$ in $\calH_0$. Below we present a suitable version of scattering theory for the triple 
$(A_0,A,J)$ under the following

\begin{assumption}\label{ass2}
\begin{enumerate}[\rm (i)]
\item
$\Ker J^*=\{0\}$. 
\item
$\Ker J\cap \Ker(A_0+T-z)=\{0\}$ 
for all $z\not=0$. 
\item
$\Ker Q_0=\Ker Q_0^*=\{0\}$. 
\item
The spectrum of $A_0$ on $\Delta$ is purely a.c.\  with a constant
multiplicity. 
The operator $Q_0$ in $\calH_0$ is strongly $A_0$-smooth on $\Delta$ 
with an exponent $\gamma\in (1/2,1)$. 
\item
The operator $T A_0$ can be factorised as 
\begin{equation}
T A_0=Q_0^*K Q_0,
\label{ccs31}
\end{equation}
where $K $ is a compact operator in $\calH_0$. 
\item
The operator $M_0=Q_0T^*Q_0^{-1}$ is bounded  on $\calH_0$ and the operator $M_0^2$  is compact. 
\item
The operator $M=Q_0J^*JQ_0^{-1}$ is bounded. 
\item
One has 
$$
A_0 (J^*J -I) A_0 =Q_0^*\wt{K}Q_0 
$$
with a compact operator  $\wt{K}$  in $\calH_0$.
\item
The operator $JA_0^2J^*- A^2$ is compact. 
\end{enumerate}
\end{assumption}

We  do not require that $\Ker J=\{0\}$; 
it turns out that instead of this it suffices to impose a much weaker Assumption~\ref{ass2}(ii).

Our spectral results are formulated in the following assertion.

\begin{theorem}\label{thm.ccs1}
Under Assumption~$\ref{ass2}$ one has:
\begin{enumerate}[\rm (i)]
\item
The a.c.\  spectrum of  the operator $A$ on $\Delta$ has the same multiplicity as that of 
the operator $A_{0}$.
\item
The singular continuous spectrum of $A$ on $\Delta$ is empty. 
\item
The eigenvalues of $A$ in $\Delta$ have finite multiplicities and 
can accumulate only to the endpoints of $\Delta$. 
\item
Let   
$$Q=Q_0J^*:\calH \to\calH_0. $$
Then the operator valued function $QR(z) Q^*$ is H\"older continuous
$($in the operator norm$)$ in $z$ for $\pm \Im z\geq 0$ and 
$\Re z\in\Delta \setminus \sigma_p(A)$. 
\end{enumerate}
\end{theorem}

Scattering theory is described in the following assertion.

\begin{theorem}\label{thm.scat}
Under Assumption~$\ref{ass2}$ one has:
\begin{enumerate}[\rm (i)]
\item
The wave operators $W_\pm(A,A_0;J,\Delta)$ exist and enjoy the intertwining property
$$
W_\pm(A,A_0;J,\Delta)A_0 = AW_\pm(A, A_0;J,\Delta).
$$

\item
The   operators $W_\pm(A,A_0;J,\Delta)$   are isometric on $\Ran E_{0}(\Delta)$ and 
are complete:
\begin{align}
W_\pm ^*(A,A_0;J,\Delta) W_\pm(A,A_0;J,\Delta)&=E_{0}(\Delta),
\label{ccs1}
\\
W_\pm(A,A_0;J,\Delta)W_\pm^*(A,A_0;J,\Delta) &=E^\ac(\Delta).
\label{ccs2}
\end{align}
\end{enumerate}
\end{theorem}

\subsection{Fredholm resolvent equations}

Let us proceed from the resolvent identity \eqref{bb3} 
which we consider as an equation for the operator $R(z)J$. 
This equation is Fredholm if 
\begin{equation}
(TR_0(z))^p\in {\goth S}_\infty \quad \text{ for some  } \;
p\in \mathbb N.
\label{ccs3}
\end{equation}
We first consider the homogeneous equation corresponding to \eqref{bb3}. 
The argument of the following lemma will be used 
several times in what follows.

\begin{lemma}\label{FR11}
Let Assumption~$\ref{ass2}\rm (ii)$ hold. 
Then for any $z$, $\Im z\not=0$, the equation 
 \begin{equation}
f+TR_0(z)f=0
\label{eq:fr2}\end{equation}
has only the trivial solution $f=0$.
\end{lemma}

\begin{proof}
Set $\varphi=R_0(z)f$. It follows from \eqref{eq:fr2} that
\begin{equation}
 (A_{0} -z)\varphi+T \varphi=0.
\label{eq:fr3}\end{equation} 
Applying the operator $J$ to this equation and using 
\eqref{bb2}, we see that $AJ\varphi=zJ\varphi$. Hence $J\varphi=0$ 
because the operator $A$ is self-adjoint. Moreover, $\varphi\in \Ker(A_0+T-z)$ according to equation \eqref{eq:fr3}.  Therefore
Assumption~\ref{ass2}(ii) implies that $\varphi=0$ whence $f=0$.
\end{proof}

In view of equation \eqref{bb3}, this directly leads to the following assertion.

\begin{lemma}\label{FRed}
If Assumption~$\ref{ass2} \rm (ii)$ and inclusion \eqref{ccs3} are satisfied, then  
\begin{equation}
R(z)J=JR_0(z)(I+TR_0(z))^{-1}, 
\quad 
\Im z\not=0,
\label{ccs4}
\end{equation}
where the inverse operator in the r.h.s.\  exists and is bounded.
\end{lemma}   

Equation \eqref{ccs4} is convenient for $\Im z\not=0$. 
In order to study the resolvent $R(z)$ when $z$ approaches
the real axis, we need to sandwich this resolvent between the operators $Q$, $Q^*$
(recall that $Q=Q_0J^*$) and rearrange equation \eqref{bb3}. This should be compared to passing from \eqref{2.0} to \eqref{2.1}
in the ``single channel'' setting. In the present case, the algebra is slightly more 
complicated. 

\begin{lemma}
Let Assumption~$\ref{ass2}${\rm(v--vii)} hold. 
Then for all $\Im z\not=0$, we have
\begin{equation}
QR(z)Q^*(I+G_0(z))
=
MQ_0R_0(z)Q_0^*,
\label{ccs6}
\end{equation}
where the operator $G_0(z)$ is given by\footnote{Formally, $G_0(z)=(Q_0^*)^{-1}TR_0(z)Q_0^*$.} 
\begin{equation}
G_0(z)=- z^{-1} M_0^*+ z^{-1}  K  Q_0 R_0(z)Q_0^*
\label{ccs7}
\end{equation}
and hence $G_0^2(z) \in {\goth S}_{\infty}$. 
\end{lemma}

\begin{proof}
First, we multiply \eqref{bb3} by $Q$ on the left and by $Q_0^*$ on the right: 
\begin{equation}
QR(z)Q^* + QR(z) J TR_0(z) Q_0^* = QJR_0(z)Q_0^* .
\label{eq:ss}\end{equation}

Let us show that
\begin{equation}
    TR_0(z) Q_0^*=Q^*_{0} G_{0}(z).
\label{eq:s}\end{equation}
Indeed, using the identity 
\begin{equation}
z R_0(z)=-I + A_0R_0(z) ,
\label{eq:RId}\end{equation}
we see that
\begin{equation}
z   TR_0(z) Q_0^*= -T  Q_0^* + (T A_{0}) R_0(z) Q_0^*. 
\label{eq:RId1}\end{equation}
By Assumption~\ref{ass2}(vi), we have
$$
T  Q_0^*=   Q_0^* M_0^* .
$$
 Using also equality \eqref{ccs31}, we can rewrite \eqref{eq:RId1} as 
$$
TR_0(z) Q_0^*= z^{-1}Q_0^*(- M_0^* +   K  Q_{0} R_0(z) Q_0^*).
$$
By definition \eqref{ccs7},  this yields \eqref{eq:s}.

Substituting \eqref{eq:s} into the second term in the l.h.s.\  of \eqref{eq:ss} and using that
$J Q^*_{0}= Q^* $,  we see that this term equals 
$QR(z)   Q^*  G_{0}(z)  $. In the r.h.s.\  of \eqref{eq:ss}, we use the fact that, by Assumption~\ref{ass2}(vii),
\begin{equation}
    QJ= Q_{0}J^*J= MQ_{0}.
\label{eq:s1}\end{equation}
   Therefore \eqref{eq:ss} yields identity \eqref{ccs6}.

The operator $G_0^2(z)$
is compact because, by Assumption~\ref{ass2}(v,vi), 
both $K$  and $M_0^2$ are compact. 
\end{proof}

The resolvent equation \eqref{ccs6} allows us to study the boundary values of $R(z)$ as $\Im z\to0$. 
It should be compared to the resolvent equation \eqref{2.1} in the ``single channel'' setting. 
The main difference between the single channel resolvent equation and 
the multichannel one is that in the single channel case, the operator
$KQ_0R_0(z)Q_0^*$ is compact, whereas under Assumption~\ref{ass2} we cannot
  guarantee the compactness of $G_0(z)$; instead, we have the 
compactness of the square $G_0^2 (z)$. 
In any case, \eqref{ccs6} is a Fredholm equation for $QR(z)Q^*$ amenable to analysis
for $\Im z\to 0$.

\begin{lemma}\label{lma.ccs3}
Let the operator $G_0(z)$ be defined by formula \eqref{ccs7}. Then
under Assumption~$\ref{ass2} {\rm(ii, iii, v-vii)}$, the 
equation
\begin{equation}
g+G_0(z)g=0, 
\quad 
\Im z\not=0,
\label{ccs8}
\end{equation}
has only the trivial solution $g=0$. 
\end{lemma}

\begin{proof}
Applying the operator $Q_0^*$ to equation \eqref{ccs8} and taking identity   \eqref{eq:s}   into account,  we obtain equation  \eqref{eq:fr2}   for $f=Q_0^*g$.     By Lemma~\ref{FR11} we have $f=0$. 
This ensures that $g=0$ because $\Ker Q_0^*=\{0\}$. 
\end{proof}

Since $G_0^2 (z)\in{\goth S}_{\infty}$, Lemma~\ref{lma.ccs3}
implies  the following assertion.

\begin{lemma}\label{lma.ccs3h}
Under the hypothesis of Lemma~$\ref{lma.ccs3}$, we have the representation
\begin{equation}
QR(z)Q^*=MQ_0R_0(z)Q_0^*(I+G_0(z))^{-1}, 
\quad 
\Im z\not=0,
\label{ccs10}
\end{equation}
where the inverse operator in the r.h.s.\  exists and is bounded. 
\end{lemma} 

\subsection{The limiting absorption principle and spectral consequences}

Let us now study equation \eqref{ccs10} as $z$ approaches the interval $\Delta$.
The first assertion is a direct consequence of
  Proposition~\ref{Privalov}.
  
\begin{lemma}\label{Priv}
Under    Assumption~$\ref{ass2}${\rm (iv)}  the operator 
valued functions $Q_0R_0(z)Q_0^*$ and hence $G_0(z)$ 
are H\"older continuous with exponent $\gamma\in (1/2,1)$ for $\Re z\in\Delta$, $\pm\Im z\geq 0$.
\end{lemma}

 Let  the  set ${\cal N}_{\pm}\subset \Delta$ consist of the points $\lambda$ such that the equation
\begin{equation}
g +  G_0(\lambda\pm i0) g=0
\label{eq:HM}
\end{equation}
has a nontrivial solution $g\not=0$.
By    Proposition~\ref{Fred}, the  set ${\cal N}_{\pm} $   is closed and has the Lebesgue measure zero. The inverse operator in the r.h.s.\  of \eqref{ccs10} is continuous for $\pm \Im z\geq 0$ away from the set ${\cal N}_{\pm}$. Equation \eqref{ccs10} implies the same statement about the operator valued function $QR  (z) Q^* $. We also take into account   that $QR  (z) Q^* $ and $QR  (\bar{z}) Q^* $ are continuous simultaneously. 
Thus, combining Proposition~\ref{Fred} and Lemma~\ref{Priv}, we obtain the 
following assertion, which is known as the limiting absorption principle.

\begin{theorem}\label{limabs}

Let  Assumption~$\ref{ass2}${\rm (ii -- vii)} hold.
Then the  set ${\cal N}_{\pm}\subset \Delta$   is closed and has the Lebesgue measure zero. The operator-valued function $QR  (z) Q^* $ is H\"older continuous with the exponent $\gamma$   up to the cut 
along $\Delta$ away from the set ${\cal N}={\cal N}_{+}\cap {\cal N}_{-}$.
\end{theorem}

Observe that the hypotheses of Theorem~\ref{limabs} do not exclude that, 
for example, $J=0$; then $Q=0$ and the statement of the theorem is vacuous. 
So in order to deduce from this theorem some spectral consequences for $A$,
we need  additional  assumptions such as $\Ker J^*=\{0\}$ 
(which is part (i) of Assumption~\ref{ass2}). 
Then the kernel of $Q=Q_0J^*$ is trivial and hence
the range of the operator $Q^*$ is dense in ${\cal H}$. 
So Theorem~\ref{limabs} implies the following result.

 \begin{corollary}\label{limabsc}
 Let Assumption~{\rm \ref{ass2}(i -- vii)} hold. 
Then, on the interval $\Delta$, the singular continuous spectrum 
and the eigenvalues of $A$ are contained in the set ${\cal N}$.
\end{corollary}

Next, we check that the ``exceptional" set $\cal N$  is exhausted by the eigenvalues 
of the operator $A$. To that end, we need to study in more detail the solutions of 
the homogeneous equation \eqref{eq:HM}. 
We start with an elementary but not quite obvious identity
which is a direct consequence of the self-adjointness of $A$.

\begin{lemma}\label{lma.ccs4}
Let Assumption~$\ref{ass2} {\rm (iii, vi, vii)}$ hold. 
For all $g\in\calH_0$ and $z=\lambda+i\eps$, $\eps\not=0$, 
we have 
\begin{equation}
\Im ((I+G_0(z))g,MQ_0R_0(z)Q_0^*g)
=
-\eps\norm{JR_0(z)Q_0^*g}^2.
\label{ccs11}
\end{equation}
\end{lemma}

\begin{proof}
By the relation \eqref{bb2}, we have 
$$
J(A_0-z)\varphi+JT\varphi=AJ\varphi-zJ\varphi
$$
for all $\varphi \in\calH_0$. 
Setting here $\varphi=R_0(z)f$, we see that 
$$
J(I+TR_0(z))f=A JR_0(z)f-zJR_0(z)f
$$
and hence
$$
(J(I+TR_0(z))f,JR_0(z)f)
=
(AJR_0(z)f,JR_0(z)f)
-z(JR_0(z)f,JR_0(z)f).
$$
Taking the imaginary part and using the self-adjointness of $A$, we get
$$
\Im (J(I+TR_0(z))f,JR_0(z)f)
=
-\eps\norm{JR_0(z)f}^2.
$$
Setting $f=Q_0^*g$, we obtain 
\begin{equation}
\Im ((I+TR_0(z))Q_0^*g,J^*JR_0(z)Q_0^*g)
=
-\eps\norm{JR_0(z) Q_0^*g }^2.
\label{ccs11x}
\end{equation}
According to identities \eqref{eq:s} and \eqref{eq:s1} the l.h.s.\  of  \eqref{ccs11} and of  \eqref{ccs11x}
 coincide.
\end{proof}

In the case $J=I$, identity \eqref{ccs11} is well known and plays the crucial role 
in the study of  the exceptional set $\calN$. 
This is still true in a more general case considered here.

Our next goal is to pass  to the limit $\varepsilon\to 0$ in \eqref{ccs11}. 
The following assertion will allow us to get rid of the operator $J$ in the r.h.s.\ 

\begin{lemma}\label{lma.ccx}
Let Assumption~$\ref{ass2} {\rm (iv, vii, viii)}$  hold. Then, for all $g\in\calH_0$, the function
$$
((J^*J-I)R_0(z)Q_0^*g,R_0(z)Q_0^*g)
$$
is continuous for $\Re z\in\Delta$, $\pm \Im z\geq 0$.
\end{lemma}

\begin{proof}
Using identity \eqref{eq:RId}, we get
\begin{multline*}
|z|^{2} R_0(\overline{z})(J^*J-I)R_0(z)
\\
=- (J^*J-I)
+ R_0(\overline{z})A_0(J^*J-I) A_0R_0(z)
- 2\Re \big(z (J^*J-I)  R_0(z)\big).
\end{multline*}
Consider separately the three terms in the r.h.s.\   The first one does not depend on $z$.
Next, by  Assumption~\ref{ass2}(viii), we have
$$
(A_0(J^*J-I)A_0R_0(z)Q_0^*g,R_0(z)Q_0^*g)
= 
(\wt K Q_0R_0(z)Q_0^*g,Q_0R_0(z)Q_0^*g). 
$$
This function is continuous because, by Lemma~\ref{Priv}, the operator valued function
$Q_0R_0(z)Q_0^*$ is continuous. 
Finally,  we have
\begin{multline*}
( (J^*J-I) R_0(z) Q_0^*g,Q_0^*g)
=
(Q_0 J^*J R_0(z)  Q_0^*g,   g) -(Q_0R_0(z)Q_0^* g, g)
\\
=(M Q_0R_0(z)Q_0^*  g, g) - (Q_0R_0(z)Q_0^* g, g),
\end{multline*}
where we have used   (vii) at the last step.
The r.h.s.\  here is again continuous because the operator valued function
  $Q_0R_0(z)Q_0^*$ is continuous. 
\end{proof}

\begin{lemma}\label{lma.ccs5}
Let Assumption~$\ref{ass2} {\rm (iii, iv, vi-viii)}$ hold. 
If $g$ satisfies equation \eqref{eq:HM} for $\lambda\in\Delta$, then 
\begin{equation}
\frac{d(E_{0}(-\infty,\lambda)Q_0^*g,Q_0^*g)}{d\lambda}=0.
\label{ccs12}
\end{equation}
\end{lemma}

\begin{proof}
Recall that according to relation \eqref{eq:ZG1} under assumption (iv) the l.h.s.\  of \eqref{ccs12} is a continuous function of $\lambda\in\Delta$. By Lemma~\ref{Priv},  the operator 
valued function  $G_0(z)$ defined by  \eqref{ccs7}
is continuous for $\Re z\in\Delta$, $\pm\Im z\geq 0$. Therefore
if $g$ satisfies equation  \eqref{eq:HM}, then 
$\norm{(I+G_0(\lambda\pm i\eps))g}\to0$ as $\eps\to+0$. 
By Lemma~\ref{lma.ccs4},  it follows that 
$$
\lim_{\eps\to+0}\eps\norm{JR_0(\lambda\pm i\eps)Q_0^*g}^2 = 0,
$$
whence, by  Lemma~\ref{lma.ccx},
\begin{equation}
\lim_{\eps\to+0}\eps\norm{ R_0(\lambda\pm i\eps)Q_0^*g}^2 = 0 .
\label{ccs14}\end{equation}
Now it remains to use the general operator theoretic identity
\begin{equation}
\frac{d(E_{0}(-\infty,\lambda)Q_0^*g,Q_0^*g)}{d\lambda}
=
\frac1\pi \lim_{\eps\to+0}\eps\norm{ R_0(\lambda\pm i\eps)Q_0^*g}^2,
\label{ccs14a}\end{equation}
which is a consequence of the relation between boundary values of a Cauchy integral and its density. Putting \eqref{ccs14} and \eqref{ccs14a} together, we get
 \eqref{ccs12}. 
\end{proof}

Given identity \eqref{ccs12}, the following two lemmas as well as the results of Subsection~4.5 are quite standard.

\begin{lemma}\label{lma.ccs6}
Let Assumption~$\ref{ass2}  {\rm (ii-viii)}$ hold. 
Then for both signs $``\pm"$ the inclusion
\begin{equation}
\calN_{\pm}\subset\sigma_p(A)\cap\Delta
\label{ccs14b}\end{equation}
is true.
\end{lemma}

\begin{proof}
Let a vector $g\not=0$ satisfy equation \eqref{eq:HM}.
Set $f=Q_0^*g$ and $\varphi=R_0(\lambda\pm i0)f$; 
let us check that $\varphi\in\calH_0$. 
Let  $F_0$ be the unitary map (see \eqref{eq:DIntF}) which diagonalises $A_0$  and   $\wh f= F_0 f$.  By assumption~(iv) (the strong smoothness of $Q_0$),
the function $\wh f (\lambda)=Z_{0}(\lambda; Q_{0}) g $  defined by \eqref{eq:ZG}
is H\"older continuous on $\Delta$  with the exponent $\gamma>1/2$. In view of equality \eqref{eq:ZG1} and Lemma~\ref{lma.ccs5},
we have  $\wh f(\lambda)=0$. Therefore 
\begin{multline*}
\norm{E_{0}(\Delta)\varphi}^2
=
\int_\Delta\abs{\mu-\lambda}^{-2}\babs{\wh f(\mu)}^2d\mu
\\
=
\int_\Delta\abs{\mu-\lambda}^{-2}\babs{\wh f(\mu)-\wh f(\lambda)}^2 d\mu
\leq 
\const \int_\Delta \abs{\mu-\lambda}^{-2+2\gamma}d\mu<\infty,
\end{multline*}
and hence $\varphi\in\calH_0$.

The following argument is quite similar to the proof of Lemma~\ref{lma.ccs3}. 
Multiplying \eqref{eq:HM} by $Q_0^*$ and using identity \eqref{eq:s}, we obtain
$$
f + T R_0(\lambda\pm i0)f=0.
$$
It follows that $\varphi=R_0(\lambda\pm i0)f$ satisfies 
\begin{equation}
(A_0-\lambda)\varphi+T\varphi=0.
\label{ccs19x}
\end{equation}
Since $AJ=J(A_0+T)$, this yields
$AJ\varphi=\lambda J\varphi$.
So it remains to check that $J\varphi\not=0$. 
Supposing the contrary and using equation \eqref{ccs19x},  
we see that $\varphi=0$, by assumption~(ii). Now it follows that $f=Q_0^*g=0$ 
and hence $g=0$. 
This contradicts the assumption $g\not=0$. 
Thus $\psi=J\varphi\not=0$ and $A\psi=\lambda\psi$. 
\end{proof}

\begin{lemma}\label{cor.ccs7}

Let Assumption~$\ref{ass2} {\rm (i-viii)}$ hold true. 
Then on the interval $\Delta$, the operator $A$ does not have 
any singular continuous spectrum. For the point spectrum, we have
\begin{equation}
\calN_{+}=\calN_{-}=\sigma_p(A)\cap\Delta.
\label{eq:NN}\end{equation}
\end{lemma}

\begin{proof}
By Corollary~\ref{limabsc}, the singular   spectrum
of $A$ on $\Delta$ is contained in $\calN$ and, in particular,
\begin{equation}
 \sigma_p(A)\cap\Delta\subset \calN.
\label{eq:NNg}\end{equation}
Since, by Lemma~\ref{lma.ccs6}, the set $\calN$ is countable,   the singular continuous spectrum of $A$ is empty. Comparing \eqref{ccs14b} with \eqref{eq:NNg},  we obtain equality \eqref{eq:NN}.  
\end{proof}

\subsection{Non-accumulation of eigenvalues}

Here we prove two results.

\begin{lemma}\label{lma.ccs9}
Let Assumption~$\ref{ass2} {\rm (i-viii)}$
 hold true. Then the eigenvalues of $A$ in $\Delta$ 
have finite multiplicities.
\end{lemma}

\begin{proof}
Taking  conjugates in \eqref{bb2}, we see that
$$
J^* A - A_{0} J^*= T^* J^*.
$$
Therefore if $A\psi=\lambda\psi$, then   the element
$\varphi=J^*\psi$ satisfies the equation 
$$
(A_0-\lambda)\varphi+T^*\varphi=0
$$
and hence the equation
$$
\varphi+R_0(\lambda\pm i0)T^*  \varphi=0.
$$
Let us apply $Q_{0}$ to the last equation and use the fact that according to \eqref{eq:s}
$ Q_{0}R_0(\lambda\pm i0)T^*  = G_0(\lambda\pm i0) Q_{0}$. Thus, for $g=Q_0 \varphi$,  we get the  equation 
$$
g+G_0^*(\lambda\pm i0)g=0.
$$

Next, we claim that $g\not=0$ if  $\psi\not=0$. 
Indeed, if $g=0$, then $\varphi=0$ because $\Ker Q_0=\{0\}$ and $\psi=0$ 
because $\Ker J^*=\{0\}$. 
Actually, the above argument   shows that 
$$
\dim \Ker (A-\lambda)
\leq 
\dim\Ker (I+ G_0^*(\lambda\pm i0) ).
$$
The dimension in the r.h.s.\  is finite 
because the operator $G_0^*(\lambda\pm i0)^2 $ is compact. 
\end{proof}
 
\begin{lemma}\label{lma.ccs8}

Let Assumption~$\ref{ass2} {\rm (i-viii)}$ hold true. 
Then the eigenvalues of $A$ in $\Delta$ 
can accumulate only to the endpoints of $\Delta$. 
\end{lemma}

\begin{proof}
Suppose, to get a contradiction, that a sequence of eigenvalues
of $A$ in $\Delta$ has an accumulation point: $\lambda_n\to\lambda_0\in \Delta$ as $n\to\infty$. 
Then by Lemma~\ref{cor.ccs7}   there exists a sequence of elements $g_n\in\calH_0$ such that
\begin{equation}
g_n+G_0(\lambda_n+i0)g_n=0, 
\quad
\norm{g_n}=1.
\label{ccs23}
\end{equation}
Since the operators $G_0(\lambda+i0)$ depend continuously on $\lambda\in\Delta$
and $G_0(\lambda+i0)^2$ are compact, we may assume that 
\begin{equation}
\norm{g_n-g_0}\to0, 
\quad
n\to\infty, 
\label{ccs24}
\end{equation}
where the element $g_0\in\calH_0$   satisfies 
$$
g_0+G_0(\lambda_0+i0)g_0=0, 
\quad
\norm{g_0}=1.
$$
Let us set 
$$
\psi_n=JR_0(\lambda_n+i0)Q_0^* g_n,
\quad
\psi_0=JR_0(\lambda_0+i0)Q_0^*g_0.
$$
By the arguments of Lemma~\ref{lma.ccs6}, it can be easily deduced from condition \eqref{ccs12} on $Q_0^* g_n$ and $Q_0^* g_0$ that
$\psi_n \in\calH_0$ and $ \psi_0\in\calH_0$. Using additionally \eqref{ccs24},  we obtain
\begin{equation}
\norm{\psi_n- \psi_0}\to 0, \quad n\to\infty. 
\label{eq:conv}\end{equation}

Exactly as in  Lemma~\ref{lma.ccs6}, equation \eqref{ccs23}   implies that $A\psi_n=\lambda_n\psi_n$, and hence $\psi_n$ are pairwise orthogonal. Therefore relation \eqref{eq:conv} can be true only if $\psi_{0}=0$. Now, again the arguments of Lemma~\ref{lma.ccs6} show that $g_{0}=0$ which contradicts the condition $\norm{g_0}=1$.
\end{proof}

Combining Theorem~\ref{limabs} 
and Lemmas~\ref{cor.ccs7} -- \ref{lma.ccs8}, 
we obtain statements (ii), (iii) and (iv) of Theorem~\ref{thm.ccs1}.

\subsection{The wave operators}
Here we prove   Theorem~\ref{thm.scat}.
Let us set $\wt J=JA_0$ and first prove intermediate results
involving the identification $\wt J$ instead of $J$.

\begin{lemma}\label{lma.ccs10}

Let Assumption~$\ref{ass2}$ hold true. 
Then the wave operators 
\begin{equation}
W_\pm(A,A_0;\wt J,\Delta),  \q
W_\pm(A_0, A;\wt J^*,\Delta)
\label{ccs28z}
\end{equation}
exist and satisfy the relations
\begin{gather}
W_\pm^*(A,A_0;\wt J,\Delta) 
=
W_\pm(A_0, A;\wt J^*,\Delta),
\label{ccs26}
\\
W_\pm ^*( A, A_0;\wt J,\Delta) W_\pm( A, A_0;\wt J,\Delta)
=
A_0^2E_{0}(\Delta),
\label{ccs27}
\\
W_\pm(A, A_0;\wt J,\Delta)W_\pm^*(A, A_0;\wt J,\Delta) 
=
A^2E^{\ac}(\Delta).
\label{ccs28}
\end{gather}
\end{lemma}

\begin{proof}
By condition \eqref{bb2} and 
Assumption~\ref{ass2}(v), we have
$$
A\wt J-\wt J A_0
=
JT A_0
=
Q^*KQ_0, \q Q=Q_{0} J^*.
$$
By Assumption~\ref{ass2}(iv), the operator $Q_0$ is strongly 
$A_0$-smooth on $\Delta$ and therefore it is 
$A_0$-smooth (in the sense of Kato) on any compact subinterval
of $\Delta$. 
Next, by Theorem~\ref{limabs}, the operator $Q$ is $A$-smooth 
(in the sense of Kato)
on every compact subinterval of the set $\Delta\setminus \calN$. 
Therefore Proposition~\ref{KL} implies the existence of the wave operators \eqref{ccs28z}; then relation \eqref{ccs26} automatically holds.

By Assumptions~\ref{ass2}(viii, ix), we have
$$
\wt J^*\wt J-A_0^2\in\goth S_\infty,
\quad
\wt J\wt J^*-A^2\in\goth S_\infty.
$$
Applying now Lemma~\ref{iso} with $\varphi(\lambda)=\lambda^2$ to the triple $A_{0}, A, \wt J$, we obtain  relation
\eqref{ccs27}. Similarly, applying   Lemma~\ref{iso} to the triple $A , A_{0}, \wt J^*$, we obtain the  relation
$$
W_\pm^*(A_0, A;\wt J^*,\Delta)W_\pm(A_0, A;\wt J^*,\Delta)
=
A^2E^{\ac}(\Delta).
$$
By \eqref{ccs26}, it is equivalent to  \eqref{ccs28}. 
\end{proof}

Now we are ready to provide 

\begin{proof}[Proof of Theorem~$\ref{thm.scat}$]
 Since, by Lemma~\ref{lma.ccs10},  the  wave operators  $W_\pm(A,A_0;\wt{J},\Delta)$ exist,   the limits \eqref{eq:WO} exist on elements $f=A_{0}g$ where $g\in{\cal H}_{0}$ is arbitrary. Using that $ {\cal H}_{0}^{\ac}\subset\overline{\Ran A_{0}}$, we see that  the  wave operators $W_\pm(A,A_0;J,\Delta)$ also exist and satisfy
\begin{equation}
W_\pm(A, A_0;J,\Delta)A_0
=
W_\pm(A, A_0;\wt J,\Delta).
\label{ccs29}
\end{equation}

It follows from \eqref{ccs27} and \eqref{ccs29} that for all $f\in{\cal H}_{0}$
$$
\| W_\pm( A,A_0;J,\Delta)A_{0} f\| = \| W_\pm( A,A_0;\wt{J},\Delta)  f\|
= \| E_{0} (\Delta) A_{0} f\|.
$$
Therefore 
\begin{equation}
\| W_\pm( A,A_0;J,\Delta)g \| =  \| E_{0} (\Delta) g\|
\label{ccs29c}
\end{equation}
for all $g\in \Ran A_{0}$ and hence for all $ g\in \overline{\Ran A_{0}}$. Every $g\in{\cal H}_{0}$ equals $g=g_{0}+g_{1}$ where $g_{0} \in \Ker A_{0}$ and $g_{1} \in \overline{\Ran A_{0}}$. Since $W_\pm( A,A_0;J,\Delta)g=W_\pm( A,A_0;J,\Delta)g_{1}$ and
$E_{0} (\Delta)  g= E_{0} (\Delta) g_{1}$, equality \eqref{ccs29c} extends to all $g\in{\cal H}_{0}$. This implies  \eqref{ccs1}.

Next, by the intertwining relation, \eqref{ccs29} yields
\begin{equation}
A W_\pm(A, A_0;J,\Delta)
=
W_\pm(A , A_0;\wt J,\Delta).
\label{ccs30}\end{equation}
Similarly, to the proof of \eqref{ccs29c}, comparing \eqref{ccs28} and \eqref{ccs30} , we see that
$$
\| W_\pm ^*( A,A_0;J,\Delta)g \| =  \| E^{\ac} (\Delta) g\|
$$
for all $g\in \Ran A $. Then, again as \eqref{ccs29c}, this equality extends to all $g\in{\cal H} $ which implies \eqref{ccs2}.
\end{proof}

Finally,  the first statement of Theorem~\ref{thm.ccs1} is a direct consequence of Theorem~\ref{thm.scat}.
The proofs of Theorems~\ref{ccs1} and \ref{thm.scat} are complete.  

\section{Proofs of Theorems~\ref{thm.b2} and \ref{thm.b2b}}\label{sec.f}
In this section we return to the setup of Section~3
and use Theorems~\ref{ccs1} and \ref{thm.scat} to prove our main results,
Theorems~\ref{thm.b2} and \ref{thm.b2b}, respectively.

Let $\calH_0=\cal H^{N+1}$, and let $A $, $A_{0}$, $J$ be given by 
\eqref{1.2a}, \eqref{1.2b}, \eqref{1.2e}, respectively. 
We set $Q_0=\diag\{X ,\dots,X \}$ in $\calH_0 $.
As usual,  $\Delta$ is a bounded open interval such that $0\not\in\Delta$.

\begin{lemma}
Let Assumption~$\ref{assu}$ be satisfied.  
Then Assumption~$\ref{ass2}$  holds true for the operators $A_{0}$, $A$ and $J$ defined above. 
\end{lemma}

\begin{proof}
(i)
is obvious because
\begin{equation}
J^* f =(f,f,\ldots,f)^\top.
\label{eq:Jstar}
\end{equation}

(ii)
According to equality \eqref{eq:AA}
$$
A_0 + T=
\begin{pmatrix}
{A_\infty} & {A_\infty} &\ldots& {A_\infty} 
\\
A_{1}& A_{1}&\ldots &A_{1}
\\
\vdots&  \vdots& \ddots&\vdots
\\
A_{N}& A_{N} & \ldots &   A_{N}
\end{pmatrix}
$$
so that for $\mathbf f=(f_\infty ,f_1,\dots,f_N)^\top\in\calH_0$
we have
$$
( A_0+T){\mathbf f}
= 
({A_\infty} J {\mathbf f}, A_{1}J {\mathbf f} ,\ldots, A_{N}J {\mathbf f} )^\top.
$$
Thus if $J {\mathbf f}=0$, then 
$(A_0+T){\mathbf f}=0$. Hence ${\mathbf f}=0$ if $J {\mathbf f}=(A_0+T-z){\mathbf f}=0$.

(iii) follows from Assumption~\ref{assu}(i).

  (iv) follows from Assumption~\ref{assu}(ii).

(v)
According to   \eqref{eq:AA}
the operator $T A_0$ has matrix entries ${A_\infty} A_\ell$ and $A_j A_\ell$, $j\not=\ell$. It follows from Assumption~\ref{assu}(iii, v) that ${A_\infty} A_\ell =X^* K_{\infty}(XA_{\ell}X^{-1}) X$ where the operators $  K_{\infty} (XA_{\ell}X^{-1}) $ are compact. For $A_j A_\ell$, we have representation \eqref{eq:m} where $  K_{j, \ell}  $ are compact.

(vi)
According to   \eqref{eq:AA}
the operator $M_0=Q_0T^*Q_0^{-1}$ 
in $\calH_0$ is represented by a matrix with entries
$X  {A_\infty} X^{-1}$ (which equals $X^* X K_{\infty}$ and is compact by Assumption~\ref{assu}(iii))
and $X A_j X^{-1}$ (which are bounded by Assumption~\ref{assu}(v)).

Similarly, the operator $M_0^2=Q_0(T^*)^2Q_0^{-1}$ is represented by a matrix
with entries $X A_j A_\ell X ^{-1}$, $j\not=\ell$, 
$X {A_\infty} A_jX ^{-1}=(X {A_\infty} X ^{-1})(X A_jX ^{-1})$, 
$X A_j{A_\infty} X ^{-1}=(X A_jX ^{-1})(X {A_\infty} X ^{-1})$ and $X A_\infty^2 X ^{-1}=
(X A_\infty  X^{-1}) ^2$. The operators $X A_j A_\ell X ^{-1}= XX^* K_{j,\ell}$, $j\not=\ell$,
are compact by Assumption~\ref{assu}(iv).  
As we have already seen, the other operators above are compact by Assumption~\ref{assu}(iii,v).

(vii)
It follows from formulas \eqref{1.2e} and \eqref{eq:Jstar} that
the operator $J^*J$ acting  in the space $\calH_0$ has the form
\begin{equation}
J^*J=
\begin{pmatrix}
I& I&  \ldots & I
\\
I& I&  \ldots & I
\\
\vdots&  \vdots& \ddots&\vdots
 \\
I& I&  \ldots & I
\end{pmatrix}.
\label{eq:star}
\end{equation}
Thus  $Q_0J^*JQ_0^{-1}=J^*J$ is a bounded operator. 

(viii)
According to formula \eqref{eq:star} the operator $A_0(J^*J-I)A_0$ in $\calH_0$ 
has matrix entries which are   zero on the diagonal and are of the form
$A_jA_\ell$ with $j\not=\ell$ off the diagonal. These operators admit     representation \eqref{eq:m} with  compact operators $K_{j,\ell}$. 

(ix) 
It follows from definitions \eqref{1.2a}, \eqref{1.2b} and
\eqref{1.2e} that 
$$
JA_0^2J^*- A^2
=
\sum_{j=1}^N A_j^2-({A_\infty} +\sum_{j=1}^N A_j)^2.
$$
In this expression the operators $A_j^2$ cancel each other. Therefore $JA_0^2J^*- A^2$
 consists of the terms  $A_\infty^2$, ${A_\infty} A_j$, $A_j{A_\infty} $ and $A_jA_\ell$, $j\not= \ell$, which are all compact by Assumption~\ref{assu} (iii, iv).
\end{proof}

Thus Theorems~\ref{ccs1} and \ref{thm.scat} are true for the operators $A_{0}$, $A$ and $J$ considered here. Theorem~\ref{thm.b2} is a direct consequence of Theorem~\ref{ccs1}. It remains to reformulate Theorem~\ref{thm.scat} as Theorem~\ref{thm.b2b}.
By the definition \eqref{1.2e}  of $J$,
the existence of the wave operators $W_\pm(A,A_0;J,\Delta)$ 
and the existence of $W_\pm(A,A_j;\Delta)$ for all $j=1,\ldots, N$ are equivalent and
\begin{equation}
W_\pm(A,A_0;J,\Delta) \mathbf f  =\sum_{j=1}^N W_\pm(A,A_{j};\Delta) f_{j}
\label{eq:star1}
\end{equation}
 if $\mathbf f =(f_{0},f_{1},\ldots,  f_N )^\top$.  The isometricity of $W_\pm(A,A_j;\Delta)$ and the intertwining property are the consequences of their existence.

Next, taking $\mathbf f =(0,\ldots,0, f_{j}, 0,\ldots,0)^\top$, $\mathbf g =(0,\ldots,0, g_{\ell}, 0,\ldots,0)^\top$ and using  \eqref{ccs1}, \eqref{eq:star1},     we obtain that
$$
(W_\pm(A,A_j;\Delta)f_j,W_\pm(A,A_\ell ;\Delta)g_\ell)=(E_{0}(\Delta)\mathbf f, \mathbf g).
$$
If $j\neq \ell$, the r.h.s.\  here is zero for arbitrary $f_{j}\in\calH$, $g_\ell \in\calH$   which implies relation \eqref{eq:Orth}.

 According to equality \eqref{ccs2} for every  $g \in\Ran E^{\ac}(\Delta)$ and $\mathbf f = W_\pm^*(A, A_0;J,\Delta) g$, we have $g = W_\pm( A, A_0;J,\Delta)\mathbf f $. Therefore, again by \eqref{eq:star1}, 
$
g=\sum_{j=1}^N W_\pm(A,A_{j};\Delta) f_{j}
$
where $\mathbf f =(0 ,f_{1},\ldots,  f_N )^\top$. This proves the asymptotic completeness \eqref{eq:AsCo}. \qed

\section{Stationary representations for wave operators and scattering matrix}
\label{sec.e}

Here we address a more special  question of stationary representations
for the wave operators and scattering matrix in the ``abstract"  framework of
Section~\ref{sec.ccs}.  Of course the representations obtained are automatically true
for the triple \eqref{1.2a}, \eqref{1.2b}, \eqref{1.2e}. This is briefly discussed in Subsection~\ref{sec.d5}.

\subsection{The scattering matrix: definition}
\label{sec.e1}

The local scattering operator for the triple $A_0$, $A$, $J$ and the interval $\Delta$
is defined by the formula 
\begin{equation}
\mathbf S(A,A_0;J,\Delta)
=
W_+(A,A_0;J,\Delta)^*W_-(A,A_0;J,\Delta).
\label{e1}
\end{equation}
By \eqref{ccs1}, \eqref{ccs2} and the intertwining property of the 
wave operators, the scattering operator is unitary on  $\Ran E_{0}(\Delta)$
and commutes with $A_0$. 

Therefore in the spectral representation (see \eqref{eq:DIntF} and \eqref{eq:FF1}) of $A_0$, the scattering operator acts as the multiplication by the operator valued function
$$
S(\lambda )=S(\lambda ;A, A_{0};  J ,\Delta):{\goth h}_{0} \to {\goth h}_{0}. 
$$
It means that 
$$
({  F}_{0}{\bf S}(A, A_{0};  J , \Delta)f )(\lambda)  = S(\lambda )
({  F}_{0}f)(\lambda), \q f\in \Ran E_{0}(\Delta) .
$$
The operator $S(\lambda )  $
is defined and is unitary for almost all $\lambda\in\Delta$. It is known as the scattering matrix. The    definition of the scattering matrix depends of course on the choice of the mapping \eqref{eq:DIntF}, but in applications the mapping ${  F}_{0} $ emerges   naturally.

Along with the scattering matrix $S(\lambda)$ corresponding to the scattering operator  
 \eqref{e1}, we consider   the scattering matrix $\wt S(\lambda)$, corresponding
to the scattering operator $\mathbf S(A,A_0;\wt J,\Delta)$ where $\wt J=J A_{0}$. 
Since
\begin{multline*}
\mathbf S(A,A_0;\wt J,\Delta)
=
W_+(A,A_0;\wt J,\Delta)^* W_-(A,A_0;\wt J,\Delta)
\\
=
A_0W_+(A,A_0;J,\Delta)^* W_-(A,A_0;J,\Delta) A_0
=
A_0\mathbf S(A,A_0;J,\Delta) A_0,
\end{multline*}
we have
\begin{equation}
\wt S(\lambda)=\lambda^2 S(\lambda).
\label{e11}
\end{equation} 
Note  that the operators $\wt S(\lambda)$ are not unitary.

\subsection{The stationary representaton for the scattering matrix}
\label{sec.e2}

Our goal here is to obtain a representaton for the scattering matrix $S(\lambda)$ in terms of the resolvent $R(z) =(A-z)^{-1}$ of the operator $A$. 
As before, we set   $Q=Q_0J^*:\calH \to\calH_0$. Recall that, by Theorem~\ref{thm.ccs1},   the operator valued function $G(z)=Q R(z) Q^*$ is H\"older continuous
$($in the operator norm$)$ in $z$ for $\pm \Im z\geq 0$ and 
$\Re z\in\Delta \setminus \sigma_p(A)$. We   also use the notation $Z_{0}(\lambda )=Z_{0}(\lambda;Q_{0})$ for operator \eqref{eq:ZG}. This operator is bounded and depends H\"older continuously on $\lambda\in \Delta$.    Now we
are ready to present the stationary representation of  $S(\lambda)$. 

\begin{theorem}\label{thm.e1}
Let Assumption~$\ref{ass2}$ hold. 
Then for all $\lambda\in\Delta\setminus \sigma_p(A)$ 
the scattering matrix $S(\lambda)$ can be represented as
\begin{equation}
S(\lambda)=I- 2\pi i \lambda^{-1} Z_{0}(\lambda)M^* K  
Z_{0}^*(\lambda) +2\pi i \lambda^{-2} Z_{0}(\lambda)  K ^* G(\lambda+i0) K  
Z_{0}^*(\lambda).
\label{e8}
\end{equation}
The operator $S(\lambda)-I$ is compact and depends H\"older continuously on $\lambda$. 
\end{theorem}

We emphasize that all operators in the r.h.s.\  of \eqref{e8} are bounded and depend H\"older continuously on $\lambda\in\Delta\setminus \sigma_p(A)$.

For the proof,  we use the fact that all the assumptions 
of the general stationary scheme (see  \cite{BY, Yafaev})  are satisfied
for the triple $A_0$, $A$, $\wt J$ (but not for $A_0$, $A$, $J$). 
Therefore we can apply the standard stationary representation for $\wt S(\lambda)$ 
(see Proposition~1 of \S~7.4 in  \cite{Yafaev}). We recall this representation at  a somewhat heuristic level. It is convenient to use a formal notation
$\Gamma_{0}(\lambda) :{\cal H}_0 \to {\goth h}_{0}$   defined by the equality
\begin{equation}
\Gamma_{0}(\lambda)f_{0}= 
( F_{0} f)(\lambda), 
 \q f\in \Ran E_{0}(\Delta) , \q \lambda\in\Delta. 
\label{e9}
\end{equation}
Observe that the operator
$$
Z_0(\lambda)=\Gamma_0(\lambda)Q_0^*
$$
is correctly defined by equality \eqref{eq:ZG}. In view of \eqref{bb2} we have
\begin{equation}
\wt{V} := A\wt{J}-\wt{J}A_{0} = JTA_0 .
\label{e7}
\end{equation}
We further observe that the auxiliary wave operator $\Omega=W_+(A_0,A_0;\wt J^*\wt J,\Delta)$ exists. It commutes with the operator $A_{0}$ and hence acts as the multiplication by the operator valued function $\Omega(\lambda)$ in the spectral representation of $A_{0}$.
 Then the representation for $\wt S(\lambda)$ formally reads as
$$
\wt S(\lambda)
=
\Omega(\lambda)  -2\pi i \Gamma_0(\lambda)
(\wt J^* \wt{V}- \wt{V}^*R(\lambda+i0)\wt{V} )
\Gamma_0^*(\lambda) .
$$

Let us show that all the terms in the r.h.s.\  are correctly defined. First, we observe that under   Assumption~\ref{ass2}(viii) $\wt J^*\wt J-A_{0}^2 \in {\goth S}_{\infty}$ so that the operator $\Omega$ exists and $\Omega=A_{0}^2 E_{0}(\Delta)$. It follows that $\Omega(\lambda)=\lambda^2 I$. 
Then we use the fact that according to \eqref{e7} and Assumption~\ref{ass2}(v, vii)
$$
\wt J^* \wt{V}=A_{0} J^* JTA_0 = A_{0}  Q_{0}^* M^* K  Q_{0}
$$
whence
$$
\Gamma_0(\lambda)\wt J^* \wt{V} \Gamma_0^*(\lambda)
= 
\lambda Z_{0}(\lambda)M^* K Z_{0}^*(\lambda).
$$
Finally, according to \eqref{e7}  and  Assumption~\ref{ass2}(v) we have $\wt{V}=Q^* K Q_{0}^*$ so that
$$
\Gamma_0(\lambda) \wt{V}^* R(\lambda+i0)\wt{V} \Gamma_0^*(\lambda)
=
Z_{0}(\lambda)  K ^* G(\lambda+i0) K  
Z_{0}^*(\lambda).
$$
Now we are in the position to formulate the  precise result.

\begin{lemma}\label{lem.e1}
Let Assumption~$\ref{ass2}$ hold. 
Then for all $\lambda\in\Delta\setminus \sigma_p(A)$ 
the scattering matrix $\wt S (\lambda)$ can be represented as
\begin{equation}
\wt S(\lambda)= \lambda^2 I-2\pi i \lambda Z_{0}(\lambda)M^* K  
Z_{0}^*(\lambda) +2\pi i Z_{0}(\lambda)  K^* G(\lambda+i0) K 
Z_{0}^*(\lambda).
\label{e8x}\end{equation}
\end{lemma}

The representation \eqref{e8} for $S(\lambda)$ directly follows from \eqref{e11} and \eqref{e8x}.
Since $Z_{0}(\lambda)$ and $G(\lambda+i0) $ depend    H\"older continuously on $\lambda$, the same is true for $S(\lambda)$. Finally, 
 the operator $S(\lambda)-I$ is compact  because by Assumption~\ref{ass2}(v) the operator $K $ is compact. 
 
\subsection{Wave operators}\label{sec.d4}

Here we briefly discuss stationary representations of the wave operators. These representations are equivalent to the expansion over appropriate generalised eigenfunctions of the operator $A$. 

Since all the assumptions of the stationary scheme of scattering theory 
are satisfied for the triple $A$, $A_0$, $\wt J$, we can directly apply 
Theorem~5.6.1 of \cite{Yafaev} to this triple. 
We use notation   \eqref{e9}  and formally set 
\begin{equation}
\wt\Gamma_\pm(\lambda)f
=
\Gamma_0(\lambda)(\wt{J}^*-  \wt{V}^* R(\lambda\pm i0))f
\label{eq:Ga}\end{equation}
for $f\in\Ran Q^*$. Similarly to the previous subsection, under Assumption~\ref{ass2} this formula acquires the correct meaning. Indeed, Assumption~\ref{ass2}(vii) shows that
$$
\Gamma_0(\lambda) \wt{J}^* Q^*= \lambda \Gamma_0(\lambda) J^*J  Q_{0}^*= \lambda \Gamma_0(\lambda) Q_{0}^* M^* =  \lambda Z_0(\lambda)   M^*
$$
and, in view of equality \eqref{e7},  Assumption~\ref{ass2}(v) shows that
\begin{align*}
\Gamma_0(\lambda) \wt{V}^* & R(\lambda\pm i0) Q^*
=   
\Gamma_0(\lambda) A_{0}T^*J^*R(\lambda\pm i0) Q^*
\\
& =   
\Gamma_0(\lambda) Q_{0}^* K ^* Q_{0}J^*R(\lambda\pm i0) Q^*
=  
Z_0(\lambda)  K ^*  G(\lambda\pm i0) .
\end{align*}
Therefore     representation \eqref{eq:Ga} can be rewritten in terms of bounded operators as
\begin{equation}
\wt\Gamma_\pm(\lambda)f
=
\lambda Z_0(\lambda)   M^*g -   Z_0(\lambda)  K ^*  G(\lambda\pm i0) g, \q f= Q^*g .
\label{eq:Gb}\end{equation}

Now we put 
$$
(\wt F_\pm f)(\lambda)
=
\wt \Gamma_\pm(\lambda)f,
\quad f\in\Ran Q^* .
$$
According to Theorem~5.6.1 of \cite{Yafaev} the operators 
$  \wt F_\pm $
extend to bounded operators from $\cal H$ to $L^2(\Delta,\goth h_0)$   and diagonalize the operator $A$: $ (\wt{F}_{\pm} Af)(\lambda) = \lambda (\wt{F}_{\pm}f)(\lambda)$.  They are related to the wave operators by the formula $ W_\pm(A,A_0;\wt J,\Delta)=\wt F_\pm^* F_0$. 
  
It remains to replace the identification $\wt J$ by $J$.
We again formally set
\begin{equation}
\Gamma_\pm(\lambda)f
=
\Gamma_0(\lambda)(J^*-   T^* J^* R(\lambda\pm i0))f
\label{eq:Gax}\end{equation}
for $f\in \Ran Q^*$. Since $\wt{J}=J A_{0}$ and $\wt{V}=J T A_{0}$, comparing \eqref{eq:Ga} and \eqref{eq:Gax} we see that $\wt\Gamma_\pm(\lambda)=\lambda \Gamma_\pm(\lambda)$. Therefore using
\eqref{eq:Gb},  we can rewrite \eqref{eq:Gax} in terms of bounded operators as
\begin{equation}
 \Gamma_\pm(\lambda)f
=
 Z_0(\lambda)   M^*g -  \lambda^{-1} Z_0(\lambda)  K ^*  G(\lambda\pm i0) g,
  \q f=Q^* g .
\label{eq:Gbx}\end{equation}
 Thus, the operators $\Gamma_\pm(\lambda): \calH\to\goth h_0$ are well defined on the dense set $\Ran Q^*$ and,  for $f\in\Ran Q^*$,  the vector valued functions $\Gamma_\pm(\lambda)f$  depend H\"older continuously on $\lambda\in\Delta\setminus \sigma_p(A)$.  Using \eqref{ccs29}, we can now rephrase the results about the wave operators $W_{\pm}(A,A_0;\wt J,\Delta)$ in terms of the wave operators $W_{\pm}(A,A_0;J,\Delta)$. This yields the following result.

\begin{theorem}\label{WOst}
    Let Assumption~$\ref{ass2}$  hold.  Define the operators $\Gamma_\pm(\lambda)$ by equation \eqref{eq:Gax} $($or, more precisely, by \eqref{eq:Gbx}$)$ and set
$$
(F_{\pm} f)(\lambda)=\Gamma_{\pm} (\lambda)f , \q f\in \Ran Q^*.
$$
Then $F_{\pm}$ extends to the partial isometry $F_{\pm}: {\cal H} \to L^2 (\Delta;  {\goth h}_{0})$ with the initial 
space $\Ran E^{\ac}(\Delta)$, and the intertwining property
$$
(F_{\pm} Af)(\lambda) = \lambda (F_{\pm}f)(\lambda)
$$ 
is satisfied. 
The operators $F_{\pm}$ and the wave operators are related by the equality
$$
 W_{\pm}(A,A_0;J,\Delta)= F_{\pm}^* F_0.
$$
Moreover, for $f_{0}\in \Ran Q_{0}^*$, $f \in \Ran Q ^*$, we have the representation
$$
( W_{\pm}(A,A_0;J,\Delta) f_{0}, f)
=     \int _{\Delta} {\pmb \langle} \Gamma_{0}(\lambda) f_{0}, \Gamma_{\pm}(\lambda) f    {\pmb \rangle} d\lambda
$$
   where $ {\pmb \langle} \cdot, \cdot    {\pmb \rangle}$ is the inner product in ${\goth h}_{0}$.
     \end{theorem}
      
\subsection{The multichannel case}\label{sec.d5}

Of course under Assumption~\ref{assu} all the results of this section are true 
for the operators $A_{0}$, $A$ and $J$ defined by \eqref{1.2a}, \eqref{1.2b} and \eqref{1.2e}. 
Let $A_{j}$, $j=1,\ldots, N$,  be realized as the operator of multiplication by   independent variable $\lambda$ in the space $L^2 (\Delta)\otimes{\goth h}_{j}$ where $\dim {\goth h_j}$ is the multiplicity of the spectrum of the operator $A_{j}$ on the interval $\Delta$. Then the scattering matrix $S(\lambda)$ is given by the matrix $S_{\ell, j}(\lambda): {\goth h_j}\to  {\goth h_k}$. Theorem~\ref{thm.e1}, in particular, shows that the operators $S_{j,j}(\lambda)-I$ and $S_{\ell, j}(\lambda) $ for $\ell\neq j$ are compact.

\appendix{}

\section{Faddeev's equations}

Let us show that Faddeev's equations for three interacting quantum particles follow from 
the resolvent equation \eqref{bb3} for a particular choice of the operators $A_{0}$, $A$ and $J$. Actually, we consider a slightly more general situation. 

Let a self-adjoint operator $H$ in a Hilbert space $\cal H$ admit the representation
\begin{equation}
H=H_{0}+\sum_{j=1}^N V_{j}.
\label{eq:QM1}\end{equation} 
 We suppose that the   operator $H_{0}$ is self-adjoint and set ${\cal R}_{0}(z)= (H_{0}-z)^{-1}$. For simplicity, we assume that all operators $V_{j}$  (but not $H_{0}$) are bounded. Our main assumption is that  
\begin{equation}
  V_j   {\cal R}_{0} (z)  V_k   \in {\goth S}_{\infty}, \q j, k =1, \ldots, N, \;  j\neq k,\q\Im z \neq 0.
\label{eq:QMcomp}
\end{equation}
In the three-particle problem $H$ is the Schr\"odinger operator, $H_{0}$ is the operator of the kinetic energy of three particles with the center-of-mass motion removed; $V_j$, $j=1,2,3$,  are potential energies of pair interactions of  particles (for example, $V_{1}$ is  the potential energy of   interaction  of the second and third particles).  

 We introduce the Hilbert space 
$\cal H_{0}=  \calH^{N}$
as the direct sum 
of $N$ copies of the space $\calH$.
The elements of this space are   columns
${\bf f}=( f_1,\dots,f_N)^\top$.
We define the operator $A_0$ in this space as
\begin{equation}
A_0=\diag\{  H_{1},\ldots, H_{N}\} \q {\rm where} \q H_{j}=H_{0}+ V_{j}
\label{eq:diag}
\end{equation}
and the operator $J:\calH_0\to\calH$ by 
\begin{equation}
J {\bf f} =\sum_{j=1}^N f_{j}.
\label{eq:QMJ}
\end{equation}
We set $A=H$.

Since
$$
(AJ -J A_{0}) {\bf f} =\sum_{j,k=1; j\neq k}^N V_{k} f_{j}
$$
 factorization \eqref{bb2}  is now true with the operator
  $T:\calH_0\to\calH_0$ acting by the formula
\begin{equation}
T=  \begin{pmatrix}
0& V_1& V_1&\ldots& V_1
\\
V_2& 0& V_2&\ldots&  V_2
 \\
 \vdots&  \vdots&\vdots& \ddots&\vdots &
  \\
 V_{N}& V_{N} & V_N&\ldots  & 0
 \end{pmatrix}.
\label{eq:QMAA}\end{equation}

Let us check Assumption~\ref{ass2}(ii).
Comparing equations \eqref{eq:diag} and \eqref{eq:QMAA} we find that 
$$
A_0 + T=
\begin{pmatrix}
H_{1}& V_1& V_1&\ldots& V_1
\\
V_2& H_{2} & V_2&\ldots&  V_{2}
 \\
 \vdots&  \vdots&\vdots& \ddots&\vdots &
  \\
 V_{N}& V_{N} & V_N&\ldots  & H_{N} 
\end{pmatrix}
$$
and hence
$$
(A_0+T){\bf f}= (H_{0}f_{1}+ V_{1} J {\bf f}, H_{0}f_{2}+V_{2} J {\bf f} ,\ldots, 
H_{0}f_{N}+V_{N} J {\bf f} )^\top.
$$
Thus, if $J {\bf f}=0$, then
$$
(A_0+ T-z){\bf f}= ((H_{0}-z)f_{1} , (H_{0}-z)f_{2} ,\ldots, 
(H_{0}-z)f_{N} )^\top.
$$
Since the operator $H_{0}$ is self-adjoint, the equality $(A_0+T-z){\bf f}=0$ implies that
$ f_{j}=0$ for all $j=  1, 2,\ldots, N$.

Next, we check inclusion \eqref{ccs3} for $p=2$. Set ${\cal R}_{j}(z) =( H_{j}-z)^{-1}$. It follows from \eqref{eq:QMAA} that
$$
T R_{0}(z)=  \begin{pmatrix}
0& V_1 {\cal R}_2 (z)& V_1 {\cal R}_3 (z)&\ldots& V_1 {\cal R}_N (z)
\\
V_2 {\cal R}_1 (z)& 0& V_2 {\cal R}_3 (z)&\ldots&  V_2 {\cal R}_N (z)
\\
\vdots&  \vdots&\vdots& \ddots&\vdots &
\\
V_N {\cal R}_1 (z)& V_N {\cal R}_2 (z) & V_N {\cal R}_3 (z)&\ldots  & 0
\end{pmatrix}.
$$
Therefore the operator $(T R_{0}(z))^2$ is given by the $N\times N$ matrix with elements
$$
V_{j} {\cal R}_k (z) V_k {\cal R}_\ell (z), \q j\neq k.
$$
By the resolvent identity applied to the pair $H_{0}$, $H_{k}$, we have
$$
V_{j} {\cal R}_k  V_k    = ( V_{j} {\cal R}_0 V_k   )
(I-     {\cal R}_k V_k  ),
$$
and hence this operator is compact by assumption \eqref{eq:QMcomp}.

 Thus Lemma~\ref{FRed} implies the following result.

\begin{theorem}\label{FRF}
Let   the operators $A_{0}$ and $A=H$ be defined by formulas \eqref{eq:diag} and \eqref{eq:QM1}, and let the operator $J$ be given by formula \eqref{eq:QMJ}. Then under assumption \eqref{eq:QMcomp} the resolvent $R(z)=(H-z )^{-1}$ of the operator $H$ admits the representation  \eqref{ccs4}
where the inverse operator on the right exists and is bounded.
\end{theorem}

Note that equation \eqref{bb3} in the case considered is equivalent to the system of Faddeev's equations. Indeed, applying both sides of \eqref{bb3} to an element ${\bf f}=(f_{1},\ldots, f_{N})^\top$, we find that 
$$
  R (z)\big( \sum_{k=1}^N f_{k}+    \sum_{j,k=1; j \neq k}^N V_{j} {\cal R}_{k} (z) f_{k}\big)= \sum_{k=1}^N {\cal R}_{k} (z) f_{k}.
$$
Since elements $f_{k}$ are arbitrary, this leads to a system of $N$ equations
\begin{equation}
R (z)\big( I+    \sum_{j =1; j \neq k}^N V_{j} {\cal R}_{k} (z)  \big)=  {\cal R}_{k} (z)  ,\q k=1,\ldots, N,
\label{eq:QMAA5}
\end{equation}
for the same object $R(z)$. We note that each of  equations \eqref{eq:QMAA5} is simply  the resolvent equation for the pair $H_{k}$, $H$. It determines the resolvent uniquely, but the operators $V_{j} {\cal R}_{k} (z)$ are not of course compact. Nevertheless, considered together, equations \eqref{eq:QMAA5}  yield the   Fredholm system.

Applying to \eqref{eq:QMAA5} on the right the operator $V_{k}$ and setting $Y_{k}(z)=R(z) V_{k}$, we obtain the system of $N$ equations  
$$
  Y_{k} (z) +    \sum_{j =1; j \neq k}^N Y_{j}(z) {\cal R}_{k} (z)V_{k}  =   {\cal R}_{k} (z) V_{k} ,\q k=1,\ldots, N,
$$
for $N$ operators $Y_{k} (z) $. 
This system was derived and used by L.~D.~Faddeev \cite{Fadd}
in the study of three particle quantum systems.

\section*{Acknowledgements} 
Our collaboration has become possible through the hospitality and financial support 
of the Departments of Mathematics of the University of Rennes 1 and of KingÕs College London. The second author was partially supported by the projects NONa (ANR-08-BLANC-0228) and NOSEVOL (ANR-11-BS0101901).



\begin{thebibliography}{00}

    \bibitem {BY}M. Sh. Birman and D. R. Yafaev, {\em A general  scheme in the stationary theory of
scattering}, Problemy Mat. Fiz. {\bf 12}  (1987), 89--117;  English transl.: Amer. Math. Soc. Transl.
(Ser. 2) {\bf 157}  (1993), 87--112.

\bibitem {Fadd} L. D. Faddeev, {\it Mathematical aspects of the three body problem in quantum scattering theory},  Trudy Mat. Inst. Steklov {\bf 69}, 1963 (Russian); English transl.: Israel Program of Sci. Transl., 1965.

   \bibitem {F2} L. D. Faddeev, {\em On the Friedrichs model in the theory of perturbations of the
continuous spectrum}, Trudy MIAN {\bf 73}  (1964), 292--313;  English transl.:  Amer. Math. Soc. Transl. (Ser. 2)  {\bf 62}  (1967), 177--203.

\bibitem{Howl0}
J.~S.~Howland,
\emph{Spectral theory   of self-adjoint Hankel matrices},
Michigan Math. J.  {\bf 33} (1986), 145--153.


 \bibitem{HoKa}
J.~S.~Howland and T. Kato,
\emph{On a theorem of Ismagilov},
J. Funct. Anal. {\bf 41} (1981),  37--39.

 \bibitem{Isma}   R. S. Ismagilov, {\em   On the spectrum of Toeplitz matrices}, Sov. Math. Dokl.  {\bf 4} (1963), 462--465.

\bibitem{Kuroda}
S.~T.~Kuroda,
\emph{Scattering theory for differential operators,}
J. Math. Soc. Japan \textbf{25} (1973)
\emph{I Operator theory,} 75--104,
\emph{II Self-adjoint elliptic operators,}
222--234.

\bibitem{Pe}
V.~V.~Peller,
\emph{Hankel operators and their applications},
Springer Verlag,  2002.


\bibitem{PY2}
A.~Pushnitski and D.~Yafaev,
\emph{Spectral theory of piecewise continuous functions of self-adjoint operators},
in preparation.



\bibitem {Su} 
A. V. Suslov, {\em On a theorem of Ismagilov}, Problemy Mat. Fiz. {\bf 9} (1979), 142--144; 
 English trasl.:  Selecta Math. Sov. {\bf 5}, no. 2   (1986), 137--139.

\bibitem{Yafaev}
D.~R.~Yafaev,
\emph{Mathematical scattering theory. General theory.}
American Mathematical Society, Providence, RI, 1992.

\bibitem {LNM}D. R. Yafaev, {\it Scattering theory: some old and new problems},  Lecture Notes Math. {\bf 1735}, Springer-Verlag, 2000.




\end{thebibliography}
\end{document}